\documentclass[12pt, reqno]{amsart}
\usepackage{amsmath,amssymb,amsthm}

\usepackage{amsmath, amssymb}
\usepackage{amsthm, amsfonts, mathrsfs}
\usepackage{mathptmx}
\usepackage{appendix}
\usepackage{fullpage}
\usepackage{amsfonts,graphicx}
\numberwithin{equation}{section}
\usepackage[colorlinks=true, pdfstartview=FitV, linkcolor=blue, citecolor=blue, urlcolor=blue]{hyperref}
\usepackage[sort,space,noadjust]{cite}

\newcommand{\R}{{\mathbb{R}}}
\def\div{ \hbox{\rm div}\,  }

\def\ta{ \theta }

\def\R{{\mathbb R}}
\def\div{ \operatorname{div}}

\def\u{\mathbf{u}}
\def\v{\mathbf{v}}
\def\w{\mathbf{w}}

\def\aa{\varphi}

\def\ddj{\dot \Delta_j}

\newcommand{\norm}[1]{\lVert #1 \rVert}
\theoremstyle{plain}

\newtheorem{theorem}{Theorem}[section]
\newtheorem{lemma}{Lemma}[section]
\newtheorem{definition}{Definition}[section]

\newtheorem{proposition}{Proposition}[section]
\newtheorem{remark}{Remark}[section]
\numberwithin{equation}{section}

\begin{document}
\title[Global well-posedness and large-time behavior for a special  $2\frac{1}{2}$D full compressible viscous non-resistive MHD system]{Global well-posedness and large-time behavior for a special  $2\frac{1}{2}$D full compressible viscous non-resistive MHD system}

	

	\author[X. Zhai]{Xiaoping Zhai}
	\address[X. Zhai]{School of Mathematics and Statistics, Guangdong University of Technology, Guangzhou, 510520, China}
	\email{pingxiaozhai@163.com}
	
	\author[S. Zhang]{Shunhang Zhang}
	\address[S. Zhang]{School of Mathematics and Computational Science, Wuyi University, Jiangmen, 529020, Guangdong, China} \email{zhangshunhang@wyu.edu.cn}
	\begin{abstract}
	In this paper, we consider the full compressible, viscous, non-resistive MHD system under the assumption that the fluids move on a plane while the magnetic field is oriented vertically. Within the framework of Besov spaces, by introducing several new unknown quantities and exploiting the intrinsic structure of the system, we prove the global well-posedness of strong solutions for initial data close to a constant equilibrium state. Furthermore, under some suitable additional conditions involving only the low-frequency part of the initial perturbation, we develop a Lyapunov-type energy argument, which yields the optimal time-decay rates of the global solution. To the best of our knowledge, our result is the first one on global solvability to the full compressible, viscous, non-resistive MHD system in multi-dimensional whole space.
	\end{abstract}
	\maketitle
	
	\section{Introduction and main results}
The motion of a compressible, viscous, non-resistive and heat-conducting magnetohydrodynamic (MHD) flow is governed by the following full compressible MHD system (see \cite[Chapter 3]{MR3222576}):
	\begin{equation}\label{mm1}
		\begin{cases}
			\partial_t\rho+\operatorname{div}(\rho \mathbf{v})=0,\\
			\rho(\partial_t\mathbf{v} + \mathbf{v}\cdot \nabla \mathbf{v}) -  \mu \Delta \mathbf{v} -(\lambda+\mu) \nabla \operatorname{div} \mathbf{v}+\nabla P=  (\nabla\times \mathbf{B})\times \mathbf{B},\\
			c_\nu\rho (\partial_t\vartheta+\mathbf{v}\cdot\nabla\vartheta)
			+ P\operatorname{div} \mathbf{v}-\kappa\Delta\vartheta=2\mu |D(\mathbf{v})|^2+\lambda(\operatorname{div}\mathbf{v})^2 ,\\
			\partial_t\mathbf{B}-\nabla\times(\mathbf{v}\times\mathbf{B})=0 ,\\
			\operatorname{div}\mathbf{B}=0.
		\end{cases}
	\end{equation}
	Here, the unknowns $\rho\in\mathbb{R}_+$, $\v\in\mathbb{R}^3$, $\vartheta\in\mathbb{R}$ and $\mathbf{B}\in\mathbb{R}^3$ represent the density of fluids, the velocity field, the temperature and the magnetic field, respectively. The pressure $P$ is given by $P:=R\rho\vartheta$, where $R>0$ is the perfect gas constant.  $D(\v):=\frac{1}{2}(\nabla\v+(\nabla\v)^\top)$ is the deformation tensor. The parameters $\mu$  and $\lambda$ are viscosity coefficients and satisfying $\mu>0$ and $\lambda+2\mu>0$. Moreover, $c_\nu>0$ is the specific heat at constant volume and $\kappa >0$ is the heat-conductivity coefficient.\par
	The MHD system is used to describe the dynamics of electrically conducting fluids in the presence of magnetic field. Its applications span a wide range of fields, including astrophysics, cosmology, and industry, among many others. The term `non-resistive' signifies that the resistivity coefficient in \eqref{mm1} is zero. Since the resistivity is inversely proportional to the electrical conductivity, the non-resistive MHD models are closely related to the conducting fluid with extremely high conductivity (see \cite{MR128226,freidberg1982ideal}).\par
	Owing to its physical importance and mathematical complexity, the theory of \eqref{mm1} has attracted continuous attention from researchers over the past decade. Compared with the compressible MHD system with both viscosity and resistivity (see e.g. \cite{MR2238891,MR2646819,MR2430634,1984Systems,MR3056749}), the study of \eqref{mm1} becomes more intricate due to the lack of magnetic diffusion, so its relevant results are less extensive and fruitful. When the temperature is a constant, \eqref{mm1} is reduced to the following isentropic compressible viscous non-resistive MHD system:\begin{equation}\label{isentropic}
		\begin{cases}
			\partial_t\rho+\operatorname{div}(\rho \mathbf{v})=0,\\
			\rho(\partial_t\mathbf{v} + \mathbf{v}\cdot \nabla \mathbf{v}) -  \mu \Delta \mathbf{v} -(\lambda+\mu) \nabla \operatorname{div} \mathbf{v}+\nabla P=  (\nabla\times \mathbf{B})\times \mathbf{B},\\
			\partial_t\mathbf{B}-\nabla\times(\mathbf{v}\times\mathbf{B})=0 ,\\
			\operatorname{div}\mathbf{B}=0,\quad P:=P(\rho).
		\end{cases}
	\end{equation}Below, we first review some previous results on \eqref{isentropic}. In one-dimensional (1D) case, the corresponding theory of \eqref{isentropic} is well-developed, particularly with regard to the global well-posedness of strong solutions for large initial data, as detailed in \cite{MR3694267,MR3983998}. In 2D case, the global existence of smooth solutions in whole space $\mathbb{R}^2$ and periodic domain $\mathbb{T}^2$ for initial data close to a constant equilibrium was obtained by Wu and Wu \cite{MR3620698} and Wu and Zhu \cite{MR4447771}, respectively. Zhong \cite{MR4094974} established the local well-posedness of strong solutions for large initial data containing vacuum. In 3D case, Tan and Wang \cite{MR3766969} proved the global existence of smooth solutions in the infinite slab $\mathbb{R}^2\times (0,1)$ by utilizing a two-tier
	energy method under the Lagrangian coordinates. Wu and Zhai \cite{MR4672311} established the global existence of smooth solutions in $\mathbb{T}^3$ near a background magnetic field. The stability and instability criteria for the stratified compressible magnetic Rayleigh–Taylor problem were investigated by Jiang and Jiang \cite{MR3895772}. If the motion of fluids takes place in
	the plane and the magnetic field acts on fluids only in the vertical direction (see \eqref{under} below), \eqref{isentropic} reduces to a $2\frac12$D compressible MHD system, whose global weak solutions for large initial data and global strong solutions for small initial data have been studied in \cite{MR3955615,MR4246170} and \cite{MR4526359,MR4752565}, respectively. For the investigations on the incompressible viscous non-resistive MHD system (i.e. the density in \eqref{isentropic} is a constant), we refer to \cite{MR3377532,MR3666563,MR3210038,MR3296601,MR3448784} and the references therein. \par
	Let us go back to \eqref{mm1}. When the variation of temperature is taken into account, the couplings and nonlinearities in the system become stronger, which makes the study more difficult. The global strong solution to the 1D initial value problem with or without vacuum was studied by Fan and Hu \cite{MR3390892} and Zhang and Zhao \cite{MR3624387}, respectively. Similar results for the 1D planar full compressible viscous non-resistive MHD system were obtained by Li \cite{MR3806077}. Li and Jiang \cite{MR4008702} investigated the global strong and weak solutions to the 1D Cauchy problem. Zhong \cite{MR4064858} constructed a local strong solution in $\mathbb{R}^2$ for large initial data containing vacuum. The global existence of smooth solutions in the infinite slab $\mathbb{R}^2\times (0,1)$ was obtained by Li \cite{MR4386439}. Li et al. \cite{MR4589732} proved the global existence of smooth solutions in $\mathbb{T}^3$ when the initial perturbation around a constant state is sufficiently small. The local well-posedness of strong solutions in three dimensions was discussed by Fan and Yu \cite{MR2451719}. As for the special $2\frac12$D full compressible MHD system \eqref{mm2}, Li and Sun \cite{MR4289027} proved the global existence of weak solutions in a bounded domain.\par
	However, to the best of our knowledge, the issue of global well-posedness of \eqref{mm1} in $\R^3$ remains unresolved, even when the initial data is sufficiently small or close to a non-zero steady solution. Therefore, we aim to make progress in this direction. Given the mathematical challenges involved, we consider a simplified scenario where the motion of fluids takes place in the plane and the magnetic field acts on fluids only in the vertical direction, namely\begin{equation}\label{under}
		\begin{split}
			&\rho:{=}\rho(t,x_1,x_2),\quad\mathbf{v}=(\mathbf{v}^1(t,x_1,x_2),\mathbf{v}^2(t,x_1,x_2),0):{=}(\mathbf{u},0),\\
			&\vartheta:{=}\vartheta(t,x_1,x_2),\quad\mathbf{B}:{=}(0,0,m(t,x_1,x_2)).
		\end{split}
	\end{equation}
	Under the above changes, \eqref{mm1} becomes the following system:\begin{equation}\label{mm2}
		\begin{cases}
			\partial_t\rho+\operatorname{div}(\rho \mathbf{u})=0,\\
			\rho(\partial_t\mathbf{u} + \mathbf{u}\cdot \nabla \mathbf{u}) -  \mu \Delta \mathbf{u} -(\lambda+\mu) \nabla \operatorname{div}\mathbf{u}+\nabla P +\dfrac{1}{2}\nabla m^2=  0,\\
			c_\nu\rho (\partial_t\vartheta+\u\cdot\nabla\vartheta)
			+ P\operatorname{div}\u-\kappa\Delta\vartheta=2\mu |D(\u)|^2+\lambda(\operatorname{div}\mathbf{u})^2 ,\\
			\partial_tm+\div(m \mathbf{u})=0 .
		\end{cases}
	\end{equation}
	Now the 3D flow is determined by two variables, so \eqref{mm2} is in fact a special $2\frac{1}{2}$D full compressible MHD system. The main purpose of this paper is to study the global well-posedness and large-time behavior of strong solutions to \eqref{mm2} near the equilibrium state $(\bar{\rho},\mathbf{0},\bar{\theta},\bar{b})$, where $\bar{\rho}$, $\bar{\theta}$ and $\bar{b}$ are positive constants. We choose to work in the framework of Besov spaces. Without loss of generality, we set $
	\bar{\rho}=\bar{\theta}=\bar{b}=\mu=c_v=R=\kappa=1,\quad \lambda=-1.
	$
	Define \begin{equation*}
		\begin{split}
			a:=\rho -1, \quad \ta:=\vartheta -1,\quad b:=m-1,
		\end{split}
	\end{equation*}
	then it follows from \eqref{mm2} that  the perturbation $(a,\mathbf{u},\theta,b)$ reads as\begin{equation}\label{mm3}
		\begin{cases}
			\partial_ta+\mathbf{u}\cdot\nabla a+\operatorname{div} \mathbf{u}=-a\operatorname{div}  \mathbf{u},\\
			\partial_t\mathbf{u}+ \mathbf{u}\cdot \nabla \mathbf{u} -  \Delta \mathbf{u}=-\dfrac{a}{1+a}\Delta \u
			-\dfrac{1}{1+a}(\nabla (a+\theta+a\theta)+\dfrac{1}{2}\nabla(b+1)^2),\\
			\partial_t\theta+\mathbf{u}\cdot\nabla\theta-\Delta\theta+ \operatorname{div} \mathbf{u}=-\theta\operatorname{div}\u-  \dfrac{a}{1+a}\Delta \theta+\dfrac{1}{1+a}(2 |D(\u)|^2-(\operatorname{div} \u)^2),\\
			\partial_t{b}+\mathbf{u}\cdot\nabla {b}+\operatorname{div} \mathbf{u}=-{b}\operatorname{div}  \mathbf{u},\\
			(a,\mathbf{u},\theta,b)|_{t=0}=(a_0,\mathbf{u}_0,\theta_0,b_0).
		\end{cases}
	\end{equation}\par
	Our main global well-posedness result of \eqref{mm3} is stated as follows.
	\begin{theorem} \label{dingli2}
		There exists a constant $\varepsilon>0$ such that for any  $(a_0^{\ell},\u_0^{\ell},\theta_0^{\ell},b_0^{\ell})\in \dot{B}_{2,1}^{0}(\R^2)$, $\u_0^h\in \dot{B}_{2,1}^{2}(\R^2),$  and   $(a_0^{h},\theta_0^{h},b_0^{h})\in \dot{B}_{2,1}^{3}(\R^2)$ satisfying\begin{equation}\label{smallness}
			\begin{split}
				\mathcal{X}_0:=\norm{(a_0,\u_0,\ta_0,{b}_0)}_{\dot{B}^{0}_{2,1}}^{\ell}
				+\norm{(a_0,\ta_0,{b}_0)}_{\dot{B}^{3}_{2,1}}^h+\norm{\u_0}_{\dot{B}^{2}_{2,1}}^h\leq \varepsilon,
			\end{split}
		\end{equation}
		the Cauchy problem \eqref{mm3} admits a unique global-in-time solution $(a,\u,\ta,b)$ with\begin{equation}\label{regularpro}
			\begin{split}
				&(a^\ell,b^\ell)\in {C}_b(\R_+;{\dot{B}}_{2,1}^{0}),\qquad\qquad\quad (a^h,b^h)\in {C}_b(\R_+;{\dot{B}}_{2,1}^{3}),\\
				&\u^\ell\in {C}_b(\R_+;{\dot{B}}_{2,1}^{0})\cap L^1(\R_+;{\dot{B}}_{2,1}^{2}),\quad \u^h\in {C}_b(\R_+;{\dot{B}}_{2,1}^{2})\cap L^1(\R_+;{\dot{B}}_{2,1}^{4}),\\
				&\theta^\ell\in {C}_b(\R_+;{\dot{B}}_{2,1}^{0})\cap L^1(\R_+;{\dot{B}}_{2,1}^{2}),\quad \theta^h\in {C}_b(\R_+;{\dot{B}}_{2,1}^{3})\cap L^1(\R_+;{\dot{B}}_{2,1}^{5}).
			\end{split}
		\end{equation}
		Furthermore, there exists some constant $C>0$ such that for all $t\geq 0$,\begin{equation}\label{xiaonorm}
			\begin{split}
				\norm{(a,\u,&{\ta},{b})}_{\widetilde{L}^{\infty}_t(\dot{B}^{0}_{2,1})}^{\ell}+\norm{(a,{\ta},{b})}_{\widetilde{L}^{\infty}_t(\dot{B}^{3}_{2,1})}^{h}
				+\norm{\u}_{\widetilde{L}^{\infty}_t(\dot{B}^{2}_{2,1})}^h\\
				&+\norm{(\aa,\u,{\ta})}_{L^1_t(\dot{B}^{2}_{2,1})}^{\ell}
				+\norm{\aa}_{L^{1}_t(\dot{B}^{3}_{2,1})}^h
				+\norm{\u}_{L^1_t(\dot{B}^{4}_{2,1})}^h+\norm{\ta}_{L^1_t(\dot{B}^{5}_{2,1})}^h\leq C\mathcal{X}_0,
			\end{split}
		\end{equation}
		where $\varphi$ is defined in \eqref{defphi}.
	\end{theorem}
	\begin{remark}
		We believe it would be possible to develop an analogue of Theorem \ref{dingli2} in the more general $L^p$ framework like those efforts for the compressible Navier-Stokes system (see e.g. \cite{MR2679372,MR2675485,MR3563240}). However, we refrain from doing so because it is beyond our primary interest in this paper, which is to figure out the basic dissipative structure of \eqref{mm2} and verify its global solvability in $\mathbb{R}^2$.
	\end{remark}
	\begin{remark}
		The stabilizing phenomenon of a background magnetic field on electrically conducting fluids has been observed in previous physical experiments and numerical simulations (see e.g. \cite{alexakis2011two,gallet2009influence,MR3377306,MR1825486,burattini2010decay}). To some extent, our result rigorously justify this phenomenon for \eqref{mm2} in mathematics.
	\end{remark}
	Let us point out the new ingredients in the proof of Theorem \ref{dingli2}. The key point of proving the global solvability of \eqref{mm3} with small initial data is to establish the uniform-in-time a priori estimates of the solution. For this, a natural idea is to use the classical compressible Navier-Stokes approach. However, we can not do it directly, as the linear part of \eqref{mm3} does not possess the same structure as the full compressible Navier-Stokes system. In order to make the compressible Navier-Stokes approach applicable, we introduce a new variable $\varphi:=a(\theta+1)+\frac{1}{2}(b+1)^2-\frac{1}{2}$ such that the system of $(\varphi,\mathbf{u},\theta)$ (see \eqref{linequ}) has the same linear part as the full compressible Navier-Stokes system. Then, by exploiting the delicate energy analysis, we can obtain a parabolic smoothing effect for low frequencies of $(\varphi,\mathbf{u},\theta)$, a damping effect for high frequencies of $\varphi$ and a parabolic smoothing effect for high frequencies of $(\mathbf{u},\theta)$, which means that such a new variable $\varphi$ captures a dissipation arising from the combination of density, magnetic field and temperature. It should be emphasized that compared with the full compressible Navier-Stokes system, the variables $(a,b)$ do not have dissipations, so we shall use some suitable product and commutator estimates to treat the nonlinear terms involving $(a,b)$. Finally, we derive the $L^\infty$ estimates of $(a,b)$ in time to enclose the a priori estimates. On the other hand, since the nonlinear part of \eqref{linequ} is different from the full compressible Navier-Stokes system, we need to choose a suitable functional setting. As mentioned above, the low frequencies of $(\varphi,\mathbf{u},\theta)$ have parabolic behavior, hence, to ensure a control of  $\|\nabla(\varphi,\mathbf{u},\theta)\|_{L^1(0,T;L^\infty)}$, by the  embedding $\dot{B}^{1}_{2,1}(\mathbb{R}^2)\hookrightarrow L^\infty(\mathbb{R}^2)$, we first impose the minimal regularity requirement on the low-frequency part, namely $(a_0^\ell,\mathbf{u}_0^\ell,\theta_0^\ell,b_0^\ell)\in\dot{B}^{0}_{2,1}(\mathbb{R}^2)$ (by the form of $\varphi$, we regard that the regularity of $\varphi$ is the same as $(a,b)$). For the high-frequency part, we assume $(a_0^h,\mathbf{u}_0^h,\theta_0^h,b_0^h)\in\dot{B}^{\alpha}_{2,1}(\mathbb{R}^2)\times\dot{B}^{\beta}_{2,1}(\mathbb{R}^2)\times\dot{B}^{\gamma}_{2,1}(\mathbb{R}^2)\times\dot{B}^{\alpha}_{2,1}(\mathbb{R}^2)$ with $\alpha\geq1$ and $\beta,\gamma\geq 0$ (to control $\|a\|_{L^\infty(0,T;L^\infty)}$ and $\|\nabla(\varphi,\mathbf{u},\theta)\|_{L^1(0,T;L^\infty)}$). Note that the term $I(a)\Delta\theta$ appears in the equation of $\varphi$ (see \eqref{linequ111}), and remember that $\varphi$ is damped and $\theta$ is parabolic at high frequencies, so the product estimate (see details in Section \ref{sec2})\begin{equation*}
		\begin{split}
			\|I(a)\Delta\theta\|_{L_T^{1}(\dot{B}^{\alpha}_{2,1})}^h\lesssim\|a\|_{L_T^{\infty}(\dot{B}^{1}_{2,1})}\|\theta\|_{L_T^{1}(\dot{B}^{\alpha+2}_{2,1})}+\|a\|_{L_T^{\infty}(\dot{B}^{\alpha}_{2,1})}\|\theta\|_{L_T^{1}(\dot{B}^{3}_{2,1})}
		\end{split}
	\end{equation*}indicates that $\gamma\geq \alpha$. Similarly, the appearance of $I(a)\nabla\varphi$ and $I(a)\Delta\theta$ in the equations of $\mathbf{u}$ and $\theta$ also yields $\alpha\geq\beta+1$ and $\alpha\geq\gamma$. Besides, the terms of type $D\mathbf{u}\otimes D\mathbf{u}$ in the equation of $\varphi$ are bounded by\begin{equation*}
		\begin{split}
			\|D\mathbf{u}\otimes D\mathbf{u}\|_{L_T^{1}(\dot{B}^{\alpha}_{2,1})}^h\lesssim\|\mathbf{u}\|_{L_T^{\infty}(\dot{B}^{2}_{2,1})}\|\mathbf{u}\|_{L_T^{1}(\dot{B}^{\alpha+1}_{2,1})},
		\end{split}
	\end{equation*}from which and the parabolic smoothing effect for the high frequencies of $\mathbf{u}$, we have $\beta\geq \max\{2,\alpha-1\}$. Therefore, a possible minimal choice of $(\alpha,\beta,\gamma)$ is $(3,2,3)$, namely $(\mathbf{u}_0^h,a_0^h,\theta_0^h,b_0^h)\in\dot{B}^{2}_{2,1}(\mathbb{R}^2)\times(\dot{B}^{3}_{2,1}(\mathbb{R}^2))^3$, which is indeed fit for the estimates of nonlinear terms. In comparision with the previous global well-posedness results for the isentropic counterpart (i.e. the temperature fluctuation $\theta$ in \eqref{mm3} is absent) in \cite{MR4526359}, in which the regularity of functional setting is at the usual critical level (see e.g. \cite{MR2679372} for the meaning of `critical'), the regularity of high-frequency part here is higher, whose main reason lies on the fact that more nonlinear terms arise in the auxiliary system \eqref{linequ} due to the presence of temperature. Thus, our result is not a straightforward generalization of that presented in \cite{MR4526359} to the non-isentropic case, but the situation we consider here is significantly more complicated.\par
	A natural next step following Theorem \ref{dingli2} is to provide a large-time asymptotic description to the global strong solution constructed above. Under the additional assumptions on the low-frequency part of the initial data, by performing a Lyapunov-type energy argument, we obtain the optimal time-decay rates of solutions as follows.
	\begin{theorem}\label{dingli3}
		Let $(a,\u,\ta,b)$ be the global solution given in Theorem \ref{dingli2}. For any $0<\sigma\leq 1$, if in addition $\|(a_0,\u_0,\ta_0,b_0)\|^{\ell}_{\dot B^{{-\sigma}}_{2,\infty}}<\infty$, then it holds, for all $t\geq 0$,  that
		\begin{equation}\label{tdecay}
			\begin{split}
				\norm{\Lambda^\gamma(\aa,\u,\ta)}_{L^2}
				&\lesssim (1+t)^{-\frac{{\sigma+\gamma}}{2}}\quad\text{if}\quad -\sigma<\gamma\leq 0,
			\end{split}
		\end{equation}where $\Lambda:=(-\Delta)^{\frac{1}{2}}$ and $\varphi$ is defined in \eqref{defphi}.
	\end{theorem}
	\begin{remark}
		The time-decay rates \eqref{tdecay} are optimal in the sense that they coincide with that of the heat semi-group.
	\end{remark}
	The rest of the paper unfolds as follows. In the Section \ref{sec2}, we obtain the a priori estimates of solutions to the Cauchy problem \eqref{mm3} and prove Theorem \ref{dingli2}. In Section \ref{sec3}, by performing a Lyapunov-type energy argument, we complete the proof of Theorem \ref{dingli3}. In the appendix, we recall some basic facts on the Littlewood-Paley theory used in this paper.\\[10pt]
	\textbf{Notation: }Throughout the paper, $C$ stands for a generic positive constant, which may vary on different lines. For brevity, we sometimes use $a\lesssim b$ to replace $a\leq Cb$. The notation $a\approx b$ means that both $a\lesssim b$ and $b\lesssim a$ hold. Given a Banach space $X$, we agree that $\|(a,\cdots,b)\|_X:=\|a\|_X+\cdots+\|b\|_X$. For any two operators $A$ and $B$, we denote by $[A,B]=AB-BA$ the commutator between $A$ and $B$. For any interval $I$ of $\mathbb{R}$, $C(I;X)$ denotes the set of continuous functions on $I$ with values in $X$. For $p\in[1,+\infty]$, we denote by $L^p(I;X)$ the set of measurable functions $u:I\rightarrow X$ such that $t\rightarrow \|u(t)\|_X$ belongs to $L^p(I)$ and write $\|\cdot\|_{L^p_T(X)}:=\|\cdot\|_{L^p(0,T;X)}$.
	
	\section{The proof of  Theorem \ref{dingli2}}\label{sec2}
	The goal of this section is to establish the global existence of strong solutions to the Cauchy problem \eqref{mm3} with small initial data. Our idea is to combine the local existence result with some a priori estimates, and then employ a standard continuity argument to extend the local solution globally. As a first step, we provide a local well-posedness result of \eqref{mm3}.
	\begin{proposition}\label{dingli1}
		Assume  $\u_0\in \dot{B}_{2,1}^{0}(\R^2)\cap \dot{B}_{2,1}^{2}(\R^2)$,  $(a_0,\theta_0,b_0)\in \dot{B}_{2,1}^{0}(\R^2)\cap \dot{B}_{2,1}^{3}(\R^2)$ with $1 + a_0$ bounded away from zero. Then there exists a positive time
		$T$ such that the Cauchy problem \eqref{mm3} admits a unique solution $(a, \u,\theta, b)$ on $[0,T]$ with $1+a$ bounded away from zero and\begin{equation}
			\begin{split}
				(a,b)&\in C([0,T ];{\dot{B}}_{2,1}^{0}\cap {\dot{B}}_{2,1}^{3}),\\
				\u&\in C([0,T];{\dot{B}}_{2,1}^{0}\cap {\dot{B}}_{2,1}^{2})\cap L^{1}
				(0,T;{\dot{B}}_{2,1}^{2}\cap {\dot{B}}_{2,1}^{4}),\\
				\theta&\in C([0,T ];{\dot{B}}_{2,1}^{0}\cap {\dot{B}}_{2,1}^{3})\cap L^{1}
				(0,T;{\dot{B}}_{2,1}^{2}\cap {\dot{B}}_{2,1}^{5}).
			\end{split}
		\end{equation}
	\end{proposition}
	\begin{proof}
		From $\eqref{mm3}$, we observe that the magnetic fluctuation $b$ not only satisfies the same equation as the density fluctuation $a$ but also has only one derivative loss in nonlinear terms. Therefore, the proof can proceed analogously to that of \cite[Theorem 0.2]{MR1855277} for the full compressible Navier-Stokes system, where $b$ can be handled in a same manner as  $a$. We omit the details for simplicity of presentation.
	\end{proof}
	With the local well-posedness result in hand, the major part in what follows is to derive the uniform-in-time a priori estimates of solutions. In order to capture the hidden dissipation of \eqref{mm3}, we introduce a new variable $\varphi$ as\begin{equation}\label{defphi}
		\begin{split}
			\varphi:=a(\theta+1)+\frac{1}{2}(b+1)^2-\frac{1}{2}.
		\end{split}
	\end{equation}It is shown by direct computations that $(\varphi,\u,\theta)$ fulfills\begin{equation}\label{linequ}
		\begin{cases}
			\partial_t\varphi+2\div \u={F_1},\\
			\partial_t\u-\Delta\u+\nabla \varphi +\nabla\ta={F_2},\\
			\partial_t\theta-\Delta\theta+ \div \u={F_3},\\
			\varphi|_{t=0}=\varphi_0:{=}a_0+a_0\theta_0+\dfrac{1}{2}b_0^2+b_0,\quad 	(\u,\ta)|_{t=0}=(\u_0,\ta_0),
		\end{cases}
	\end{equation}
	where
	\begin{equation}\label{linequ111}
		\begin{split}
			{F_1}:=&-\u\cdot\nabla \varphi-2\varphi\div \u-\theta\div\u+ I(a)\Delta\theta+I(a)(2|D(\u)|^2-(\div \u)^2),\\
			{F_2}:=&-\u\cdot\nabla \u +I(a)\nabla\varphi+I(a)\nabla\theta-I(a)\Delta\u,\\
			{F_3}:=&-\u\cdot\nabla \theta-\theta\div\u- I(a)\Delta\theta+\dfrac{1}{1+a}(2|D(\u)|^2-(\div \u)^2),
		\end{split}
	\end{equation}with\begin{equation*}
		\begin{split}
			I(a):=\dfrac{a}{1+a}.
		\end{split}
	\end{equation*}\par
	Throughout this section, we take for granted that
	\begin{equation}\label{eq:smalla}
		\sup_{t\in\R_+,\, x\in\R^2} |a(t,x)|\leq \frac12,
	\end{equation}
	which allows us to use freely the composition estimates stated in Lemma \ref{fuhe}. Actually,
	since $\dot B^{1}_{2,1}(\R^2)\hookrightarrow L^\infty(\R^2)$, condition \eqref{eq:smalla} is guaranteed by the small solution bound in \eqref{xiaonorm}. Besides,
	to simplify the writing, we adopt the following notations:\begin{equation}\label{normdef}
		\begin{split}
			\mathcal{E}(t):=&
			\norm{(\varphi,\u,{\ta},a,{b})(t)}_{\dot{B}^{0}_{2,1}}^{\ell}+\norm{(\varphi,{\ta},a,{b})(t)}_{\dot{B}^{3}_{2,1}}^{h}
			+\norm{\u(t)}_{\dot{B}^{2}_{2,1}}^h,\\
			\mathcal{D}(t):=&\norm{(\aa,\u,{\ta})(t)}_{\dot{B}^{2}_{2,1}}^{\ell}
			+\norm{\aa(t)}_{\dot{B}^{3}_{2,1}}^h
			+\norm{\u(t)}_{\dot{B}^{4}_{2,1}}^h+\norm{\ta(t)}_{\dot{B}^{5}_{2,1}}^h.
		\end{split}
	\end{equation}
	
	\subsection{Low-frequency estimates of $(\aa,{\u},{\ta})$}
	In this subsection, we establish the dissipative estimates of $(\aa,{\u},{\ta})$ in the low-frequency region.
	\begin{lemma}\label{dipinlemma}
		For any $t>0$, it holds that \begin{equation}\label{han200}
			\begin{split}
				&\|(\aa,{\u},{\ta})\|_{\widetilde{L}^{\infty}_t(\dot{B}^{0}_{2,1})}^{\ell}
				+\|(\aa,{\u},{\ta})\|_{L^1_t(\dot{B}^{2}_{2,1})}^{\ell}\lesssim\|(\aa_0,{\u}_0,{\ta}_0)\|_{\dot{B}^{0}_{2,1}}^{\ell}
				+\int_{0}^{t}(1+\mathcal{E}(\tau))\mathcal{E}(\tau)\mathcal{D}(\tau)\,d\tau.
			\end{split}
		\end{equation}
	\end{lemma}
	\begin{proof}
		For any $j\in\mathbb{Z}$, applying the operator $\dot{\Delta}_j$ to $\eqref{linequ}$, we get\begin{equation}\label{frelocal}
			\begin{cases}
				\partial_t\dot{\Delta}_j\varphi+2\div \dot{\Delta}_j\u=\dot{\Delta}_j{F_1},\\
				\partial_t\dot{\Delta}_j\u-\Delta\dot{\Delta}_j\u+\nabla\dot{\Delta}_j \varphi +\nabla\dot{\Delta}_j\ta=\dot{\Delta}_j{F_2},\\
				\partial_t\dot{\Delta}_j\theta-\Delta\dot{\Delta}_j\theta+ \div \dot{\Delta}_j\u=\dot{\Delta}_j{F_3}.
			\end{cases}
		\end{equation}
		Taking the $L^2$ inner product of $\eqref{frelocal}_1$, $\eqref{frelocal}_2$ and $\eqref{frelocal}_3$ with $\frac{1}{2}\ddj \varphi$, $\ddj \u$ and $\ddj \ta$ respectively, and adding up the resulting equations, we have\begin{equation}\label{qi4}
			\begin{split}
				&\frac 12 \frac{d}{dt}(\frac{1}{2}\|\ddj \varphi\|^2_{L^2}+\|\ddj \u\|^2_{L^2}+\|\ddj \theta\|^2_{L^2} )+\|\nabla\ddj \u\|^2_{L^2}+\|\nabla\ddj \theta\|^2_{L^2}\\
				&\quad=\frac{1}{2}\int_{\R^2}\ddj {F_1}\cdot \ddj \varphi\,dx+\int_{\R^2}\ddj {F_2}\cdot \ddj \u\,dx+\int_{\R^2}\ddj {F_3}\cdot \ddj \theta\,dx,
			\end{split}
		\end{equation}where we have used the following cancellation:\begin{equation}
			\begin{split}
				\int_{\R^2} \div\ddj \u\cdot \ddj \varphi\,dx+\int_{\R^2}\div\ddj  \u\cdot \ddj \theta\,dx+\int_{\R^2}\nabla\ddj(\varphi+\theta)\cdot \ddj \u\,dx=0.
			\end{split}
		\end{equation}
		In order to obtain the dissipation of $\varphi$, testing $\eqref{frelocal}_1$  by $-\div\ddj \u$ and $\eqref{frelocal}_2$  by $\nabla\ddj \varphi$, and using the integration by parts, we get\begin{equation}\label{qi7}
			\begin{split}
				& \frac{d}{dt}\left(\int_{\R^2}\ddj  \u\cdot \nabla \ddj  \varphi\,dx\right)+\|\nabla \ddj  \varphi\|^2_{L^2}-2\|\ddj \div \u\|^2_{L^2}+\int_{\R^2}(\nabla\ddj\theta -\Delta\dot{\Delta}_j\u)\cdot \nabla \ddj  \varphi\,dx\\
				&\quad=\int_{\R^2}\nabla\ddj{F_1}\cdot \ddj \u\,dx+\int_{\R^2}\ddj {F_2}\cdot \nabla\ddj \varphi\,dx.
			\end{split}
		\end{equation}Multiplying \eqref{qi7} by a small constant $\eta_1>0$ (to be determined) and then adding to \eqref{qi4}, we are led to\begin{equation}\label{qi10}
			\begin{split}
				\frac{d}{dt}{\mathcal{L}}_{\ell,j}^2+\widetilde{\mathcal{L}}_{\ell,j}^2=&\frac{1}{2}\int_{\R^2}\ddj {F_1}\cdot \ddj \varphi\,dx+\int_{\R^2}\ddj {F_2}\cdot \ddj \u\,dx+\int_{\R^2}\ddj {F_3}\cdot \ddj \theta\,dx\\
				&-\eta_1\Big(\int_{\R^2}\nabla\ddj{F_1}\cdot \ddj \u\,dx+\int_{\R^2}\ddj {F_2}\cdot \nabla\ddj \varphi\,dx\Big),
			\end{split}
		\end{equation}where\begin{equation}
			\begin{split}
				\mathcal{L}_{\ell,j}^2&:=\frac{1}{4}\|\ddj \varphi\|^2_{L^2}+\frac{1}{2}\|\ddj\u\|^2_{L^2}+\frac{1}{2}\|\ddj\theta\|^2_{L^2}+\eta_1\int_{\R^2}\ddj  \u\cdot \nabla \ddj  \varphi\,dx,\\
				\widetilde{\mathcal{L}}_{\ell,j}^2&:=\|\nabla\ddj\u\|^2_{L^2}+\|\nabla\ddj\theta\|^2_{L^2}\\&\qquad+\eta_1\Big(\|\nabla \ddj  \varphi\|^2_{L^2}-2\|\ddj \div \u\|^2_{L^2}+\int_{\R^2}(\nabla\ddj\theta -\Delta\dot{\Delta}_j\u)\cdot \nabla \ddj  \varphi\,dx\Big).
			\end{split}
		\end{equation}
		According to Berstein's inequality, we can choose $\eta_1$ sufficiently small such that for any $j\leq 0$,\begin{equation}
			\begin{split}
				\mathcal{L}_{\ell,j}^2&\approx \|(\dot{\Delta}_j\varphi,\dot{\Delta}_j\u,\dot{\Delta}_j\theta)\|_{L^2}^2,\\
				\widetilde{\mathcal{L}}_{\ell,j}^2&\approx \|(\nabla\dot{\Delta}_j\varphi,\nabla\dot{\Delta}_j\u,\nabla\dot{\Delta}_j\theta)\|_{L^2}^2\approx 2^{2j}\|(\dot{\Delta}_j\varphi,\dot{\Delta}_j\u,\dot{\Delta}_j\theta)\|_{L^2}^2.
			\end{split}
		\end{equation}
		Therefore, using H\"{o}lder's inequality in \eqref{qi10} gives that  for any $j\leq 0$,
		\begin{equation}
			\frac{d}{dt}{\mathcal{L}}_{\ell,j}^2+2^{2j}{\mathcal{L}}_{\ell,j}^2\lesssim \|(\ddj {F_1},\ddj {F_2},\ddj {F_3})\|_{L^2}{\mathcal{L}}_{\ell,j},
		\end{equation}whence, dividing by ${\mathcal{L}}_{\ell,j}$ and integrating on $[0,t]$,\begin{equation}\label{qi12}
			\begin{split}
				{\mathcal{L}}_{\ell,j}(t)+2^{2j}\int_{0}^{t}{\mathcal{L}}_{\ell,j}\,d\tau\lesssim  {\mathcal{L}}_{\ell,j}(0)+\int_{0}^{t}\|(\ddj {F_1},\ddj {F_2},\ddj {F_3})\|_{L^2}\,d\tau.
			\end{split}
		\end{equation}
		Finally, taking the supremum of \eqref{qi12} on $[0,t]$, and
		summing up on $j\leq 0$, we can conclude that\begin{equation}\label{ping3-1}
			\begin{split}
				&\|(\aa,{\u},{\ta})\|_{\widetilde{L}^{\infty}_t(\dot{B}^{0}_{2,1})}^{\ell}
				+\|(\aa,{\u},{\ta})\|_{L^1_t(\dot{B}^{2}_{2,1})}^{\ell}\lesssim\|(\aa_0,{\u}_0,{\ta}_0)\|_{\dot{B}^{0}_{2,1}}^{\ell}
				+\int_{0}^{t}\|({F_1},{F_2},{F_3})\|_{\dot{B}^{0}_{2,1}}^{\ell}\,d\tau.
			\end{split}
		\end{equation}\par
		Next, we bound the nonlinear terms in the right-hand side above. From \eqref{lhembedding} and \eqref{normdef}, one may check that\begin{equation}\label{jynl}
			\begin{cases}
				\|(\varphi,\theta,a,b)(t)\|_{\dot{B}^{\alpha}_{2,1}}\lesssim \mathcal{E}(t)\quad \text{for}\quad 0\leq\alpha\leq 3,\\
				\|\u(t)\|_{\dot{B}^{\alpha}_{2,1}}\lesssim \mathcal{E}(t)\quad \text{for}\quad 0\leq\alpha\leq 2,\qquad
				\|\varphi(t)\|_{\dot{B}^{\alpha}_{2,1}}\lesssim \mathcal{D}(t)\quad \text{for}\quad 2\leq\alpha\leq 3,\\
				\|\u(t)\|_{\dot{B}^{\alpha}_{2,1}}\lesssim \mathcal{D}(t)\quad \text{for}\quad 2\leq\alpha\leq 4,\qquad
				\|\theta(t)\|_{\dot{B}^{\alpha}_{2,1}}\lesssim \mathcal{D}(t)\quad \text{for}\quad 2\leq\alpha\leq 5.
			\end{cases}
		\end{equation}To simplify calculations, we will repeatedly use the estimates in \eqref{jynl} to control the nonlinear terms. Resorting to \eqref{daishu} and \eqref{jynl}, we have\begin{equation}\label{jiayou1}
			\begin{split}
				\|(\u\cdot\nabla\varphi,\u\cdot\nabla\u,\u\cdot\nabla\theta)\|_{\dot{B}^{0}_{2,1}}^{\ell}&\lesssim \|\u\|_{\dot{B}^{0}_{2,1}}\|(\varphi,\u,\theta)\|_{\dot{B}^{2}_{2,1}}\lesssim \mathcal{E}(\tau)\mathcal{D}(\tau),\\
				\|(\varphi\div \u,\theta\div \u)\|_{\dot{B}^{0}_{2,1}}^{\ell}&\lesssim \|(\varphi,\theta)\|_{\dot{B}^{0}_{2,1}}\|\u\|_{\dot{B}^{2}_{2,1}}\lesssim \mathcal{E}(\tau)\mathcal{D}(\tau).
			\end{split}
		\end{equation}
		To handle the composition function $I(a)$, we infer from Lemma \ref{fuhe} and \eqref{eq:smalla} that\begin{equation}\label{ag1}
			\begin{split}
				\|I(a)\|_{\dot{B}^\alpha_{2,1}}\lesssim \|a\|_{\dot{B}^\alpha_{2,1}}\quad\text{for}\quad0<\alpha\leq 3.
			\end{split}
		\end{equation}Due to the fact that $I(a)=a-aI(a)$, using \eqref{daishu} in addition, we also find that\begin{equation}
			\begin{split}
				\|I(a)\|_{\dot{B}^0_{2,1}}&\lesssim\|a\|_{\dot{B}^0_{2,1}}+\|a\|_{\dot{B}^0_{2,1}}\|I(a)\|_{\dot{B}^1_{2,1}}\\
				&\lesssim (1+\|a\|_{\dot{B}^1_{2,1}})\|a\|_{\dot{B}^0_{2,1}}.
			\end{split}
		\end{equation}As a result, applying \eqref{daishu} and \eqref{jynl} to the nonlinear terms involving $I(a)$, we obtain\begin{equation}\label{jiayou2}
			\begin{split}
				\|(I(a)\Delta\theta,I(a)\Delta\u)\|_{\dot{B}^{0}_{2,1}}^{\ell}&\lesssim \|I(a)\|_{\dot{B}^{1}_{2,1}}\|(\theta,\u)\|_{\dot{B}^{2}_{2,1}}\\
				&\lesssim \|a\|_{\dot{B}^{1}_{2,1}}\|(\theta,\u)\|_{\dot{B}^{2}_{2,1}}\lesssim \mathcal{E}(\tau)\mathcal{D}(\tau),\\
				\|(I(a)\nabla\varphi,I(a)\nabla\theta)\|_{\dot{B}^{0}_{2,1}}^{\ell}&\lesssim \|I(a)\|_{\dot{B}^{0}_{2,1}}\|(\varphi,\theta)\|_{\dot{B}^{2}_{2,1}}\\
				&\lesssim (1+\|a\|_{\dot{B}^{1}_{2,1}})\|a\|_{\dot{B}^{0}_{2,1}}\|(\varphi,\theta)\|_{\dot{B}^{2}_{2,1}}\\
				&\lesssim (1+\mathcal{E}(\tau))\mathcal{E}(\tau)\mathcal{D}(\tau),\\
				\|I(a)(2|D(\u)|^2&-(\div \u)^2)\|_{\dot{B}^{0}_{2,1}}^{\ell}+\Big\|\frac{1}{1+a}(2|D(\u)|^2-(\div \u)^2)\Big\|_{\dot{B}^{0}_{2,1}}^{\ell}\\
				&\lesssim (1+\|I(a)\|_{\dot{B}^{1}_{2,1}})\|D\u\otimes D\u\|_{\dot{B}^{0}_{2,1}}\\
				&\lesssim (1+\|a\|_{\dot{B}^{1}_{2,1}})\|\u\|_{\dot{B}^{1}_{2,1}}\|\u\|_{\dot{B}^{2}_{2,1}}\\
				&\lesssim (1+\mathcal{E}(\tau))\mathcal{E}(\tau)\mathcal{D}(\tau).
			\end{split}
		\end{equation}
		Plugging \eqref{jiayou1} and \eqref{jiayou2} into \eqref{ping3-1}, we end up with the desired estimate \eqref{han200}.
	\end{proof}
	
	\subsection{High-frequency estimates of $(\aa,{\u},{\ta})$} The dissipative estimates of $(\aa,{\u},{\ta})$ in the high-frequency region are derived in this subsection.
	\begin{lemma}\label{gaopinlemma}
		For any $t>0$, it holds that \begin{equation}\label{han201}
			\begin{split}
				&\|({\aa},\mathbf{\ta})\|_{\widetilde{L}^{\infty}_t(\dot{B}^{3}_{2,1})}^h+\|\u\|_{\widetilde{L}^{\infty}_t(\dot{B}^
					{2}_{2,1})}^h+\|{\aa}\|_{{L}^{1}_t(\dot{B}^{3}_{2,1})}^h
				+\|\u\|_{{L}^{1}_t(\dot{B}^{4}_{2,1})}^h+\|\mathbf{\ta}\|_{L^1_t(\dot B^{5}_{2,1})}^h\\
				&\quad\lesssim\|(\aa_0,{\ta}_0)\|_{\dot{B}^{3}_{2,1}}^{h}+\|{\u}_0\|_{\dot{B}^{2}_{2,1}}^{h}+\int_{0}^{t}(1+\mathcal{E}(\tau))\mathcal{E}(\tau)\mathcal{D}(\tau)\,d\tau.
			\end{split}
		\end{equation}
	\end{lemma}
	\begin{proof}
		By using a commutator argument to $\eqref{frelocal}_1$ and  $\eqref{frelocal}_2$, one may see that\begin{equation}\label{newform}
			\begin{cases}
				\partial_t\nabla\dot{\Delta}_j\varphi+2\nabla\div \dot{\Delta}_j\u=-\u\cdot\nabla\nabla\dot{\Delta}_j\varphi-[\dot{\Delta}_j,\u\cdot\nabla]\nabla\varphi-\dot{\Delta}_j(\nabla\u\cdot\nabla\varphi)+\nabla\dot{\Delta}_j{F_4},\\
				\partial_t\dot{\Delta}_j\u-\Delta\dot{\Delta}_j\u+\nabla\dot{\Delta}_j \varphi +\nabla\dot{\Delta}_j\ta=-\u\cdot\nabla\dot{\Delta}_j\u-[\dot{\Delta}_j,\u\cdot\nabla]\u+\dot{\Delta}_j{F_5},
			\end{cases}
		\end{equation}where\begin{equation}\label{jihao}
			\begin{split}
				{F_4}&:=-2\varphi\div \u-\theta\div\u+ I(a)\Delta\theta+I(a)(2|D(\u)|^2-(\div \u)^2),\\
				{F_5}&:=I(a)\nabla\varphi+I(a)\nabla\theta-I(a)\Delta\u.
			\end{split}
		\end{equation}
		By means of a standard energy argument, it is easy to show from \eqref{newform} that\begin{equation}\label{qi888}
			\begin{split}
				& \frac{d}{dt}\left(\frac{1}{4}\|\nabla\dot{\Delta}_j\varphi\|^2_{L^2}+\int_{\R^2}\ddj  \u\cdot \nabla \ddj  \varphi\,dx\right)+\|\nabla \ddj  \varphi\|^2_{L^2}\\&\quad-2\|\ddj \div \u\|^2_{L^2}+\int_{\R^2}\nabla\ddj\theta \cdot \nabla \ddj  \varphi\,dx\\
				&=\int_{\R^2}(\dot{\Delta}_j\u\cdot\nabla\dot{\Delta}_j\varphi+\frac{1}{2}|\nabla\dot{\Delta}_j\varphi|^2)\div \u \, dx-\int_{\R^2}[\dot{\Delta}_j,\u\cdot\nabla]\nabla\varphi\cdot (\dot{\Delta}_j\u+\frac{1}{2}\nabla\dot{\Delta}_j\varphi)\,dx\\
				&\quad-\int_{\R^2}[\dot{\Delta}_j,\u\cdot\nabla]\u\cdot\nabla\dot{\Delta}_j\varphi\,dx-\int_{\R^2}\dot{\Delta}_j(\nabla\u\cdot\nabla\varphi)\cdot (\dot{\Delta}_j\u+\frac{1}{2}\nabla\dot{\Delta}_j\varphi)\,dx\\
				&\quad+\int_{\R^2}\nabla\dot{\Delta}_j{F}_4\cdot(\dot{\Delta}_j\u+\frac{1}{2}\nabla\dot{\Delta}_j\varphi)\,dx+\int_{\R^2}\dot{\Delta}_j{F}_5\cdot\nabla\dot{\Delta}_j\varphi\,dx,
			\end{split}
		\end{equation}
		where we have used the following integration by parts:\begin{equation}
			\begin{split}
				&\int_{\R^2}\nabla\div \dot{\Delta}_j\u\cdot \nabla\dot{\Delta}_j\varphi\,dx=\int_{\R^2}\Delta\dot{\Delta}_j\u\cdot\nabla\dot{\Delta}_j\varphi\,dx,\\
				&\int_{\R^2}(\u\cdot\nabla\nabla\dot{\Delta}_j\varphi)\cdot\nabla\dot{\Delta}_j\varphi\,dx=-\frac{1}{2}\int_{\R^2}|\nabla\dot{\Delta}_j\varphi|^2\div \u \, dx,\\
				&\int_{\R^2}(\u\cdot\nabla\nabla\dot{\Delta}_j\varphi)\cdot\dot{\Delta}_j\u\,dx+\int_{\R^2}(\u\cdot\nabla\dot{\Delta}_j\u)\cdot\nabla\dot{\Delta}_j\varphi\,dx=-\int_{\R^2}(\dot{\Delta}_j\u\cdot\nabla\dot{\Delta}_j\varphi)\div \u \, dx.
			\end{split}
		\end{equation}
		Multiplying \eqref{qi888} by a small constant $\eta_2>0$ (to be determined) and then adding to \eqref{qi4}, we have\begin{equation}\label{qi999}
			\begin{split}
				\frac{d}{dt}&{\mathcal{L}}_{h,j}^2+\widetilde{\mathcal{L}}_{h,j}^2=\frac{1}{2}\int_{\R^2}\ddj {F_1}\cdot \ddj \varphi\,dx+\int_{\R^2}\ddj{F_2}\cdot \ddj \u\,dx+\int_{\R^2}\ddj{F_3}\cdot \ddj \theta\,dx\\
				&+\eta_2\Big(\int_{\R^2}(\dot{\Delta}_j\u\cdot\nabla\dot{\Delta}_j\varphi+\frac{1}{2}|\nabla\dot{\Delta}_j\varphi|^2)\div \u \, dx-\int_{\R^2}[\dot{\Delta}_j,\u\cdot\nabla]\nabla\varphi\cdot (\dot{\Delta}_j\u+\frac{1}{2}\nabla\dot{\Delta}_j\varphi)\,dx\\
				&\quad-\int_{\R^2}[\dot{\Delta}_j,\u\cdot\nabla]\u\cdot\nabla\dot{\Delta}_j\varphi\,dx-\int_{\R^2}\dot{\Delta}_j(\nabla\u\cdot\nabla\varphi)\cdot (\dot{\Delta}_j\u+\frac{1}{2}\nabla\dot{\Delta}_j\varphi)\,dx\\
				&\quad+\int_{\R^2}\nabla\dot{\Delta}_j{F}_4\cdot(\dot{\Delta}_j\u+\frac{1}{2}\nabla\dot{\Delta}_j\varphi)\,dx+\int_{\R^2}\dot{\Delta}_j{F}_5\cdot\nabla\dot{\Delta}_j\varphi\,dx\Big),
			\end{split}
		\end{equation}where\begin{equation}
			\begin{split}
				\mathcal{L}_{h,j}^2&:=\frac{1}{4}\|\ddj \varphi\|^2_{L^2}+\frac{1}{2}\|\ddj\u\|^2_{L^2}+\frac{1}{2}\|\ddj\theta\|^2_{L^2}+\frac{\eta_2}{4}\|\nabla\ddj \varphi\|^2_{L^2}+\eta_2\int_{\R^2}\ddj  \u\cdot \nabla \ddj  \varphi\,dx,\\
				\widetilde{\mathcal{L}}_{h,j}^2&:=\|\nabla\ddj\u\|^2_{L^2}+\|\nabla\ddj\theta\|^2_{L^2}\\&\qquad+\eta_2\Big(\|\nabla \ddj  \varphi\|^2_{L^2}-2\|\ddj \div \u\|^2_{L^2}+\int_{\R^2}\nabla\ddj\theta \cdot \nabla \ddj  \varphi\,dx\Big).
			\end{split}
		\end{equation}Thanks to Berstein's inequality, taking $\eta_2$ small enough can ensure that for any $j\geq -1$,\begin{equation}
			\begin{split}
				\mathcal{L}_{h,j}^2
				&\approx \|(\nabla\dot{\Delta}_j\varphi,\dot{\Delta}_j\u,\dot{\Delta}_j\theta)\|_{L^2}^2,\\
				\widetilde{\mathcal{L}}_{h,j}^2&  \approx \|(\nabla\dot{\Delta}_j\varphi,\nabla\dot{\Delta}_j\u,\nabla\dot{\Delta}_j\theta)\|_{L^2}^2\gtrsim \|(\nabla\dot{\Delta}_j\varphi,\dot{\Delta}_j\u,\dot{\Delta}_j\theta)\|_{L^2}^2.
			\end{split}
		\end{equation}
		Consequently, using H\"{o}lder's inequality in \eqref{qi999} indicates that for any $j\geq -1$,\begin{equation}\label{qi123}
			\begin{split}
				\frac{d}{dt}{\mathcal{L}}_{h,j}^2+{\mathcal{L}}_{h,j}^2\lesssim\Big(& \|\div\u\|_{L^\infty}\|\nabla\dot{\Delta}_j\varphi\|_{L^2}+\|\dot{\Delta}_j(\nabla\u\cdot\nabla\varphi)\|_{L^2}+\|[\dot{\Delta}_j,\u\cdot\nabla]\nabla\varphi\|_{L^2}\\
				&+\|[\dot{\Delta}_j,\u\cdot\nabla]\u\|_{L^2}+\|(\ddj {F_1},\ddj {F_2},\ddj {F_3},\nabla\ddj {F_4},\ddj {F_5})\|_{L^2}\Big){\mathcal{L}}_{h,j}.
			\end{split}
		\end{equation}
		Following similar arguments as in \eqref{qi12}-\eqref{ping3-1}, and using the fact that $F_1=F_4+\u\cdot\nabla\varphi$ and $\|({F_3},{F_4},{F_5})\|_{\dot{B}^{2}_{2,1}}^{h}\lesssim \|({F_3},{F_4},{F_5})\|_{\dot{B}^{3}_{2,1}}^{h}$, we can deduce\begin{equation}\label{estil1}
			\begin{split}
				\|\aa&\|_{\widetilde{L}^{\infty}_t(\dot{B}^{3}_{2,1})}^{h}+\|({\u},{\ta})\|_{\widetilde{L}^{\infty}_t(\dot{B}^{2}_{2,1})}^{h}
				+\|\aa\|_{L^1_t(\dot{B}^{3}_{2,1})}^{h}+\|({\u},{\ta})\|_{L^1_t(\dot{B}^{2}_{2,1})}^{h}\\
				&\lesssim\|\aa_0\|_{\dot{B}^{3}_{2,1}}^{h}+\|({\u}_0,{\ta}_0)\|_{\dot{B}^{2}_{2,1}}^{h}+\mathcal{N}(t),
			\end{split}
		\end{equation}where\begin{equation}
			\begin{split}
				\mathcal{N}(t):=&\int_{0}^{t}\|\nabla\u\|_{L^\infty}\|\varphi\|_{\dot{B}^{3}_{2,1}}^h\,d\tau+\int_{0}^{t}\|\nabla\u\cdot\nabla\varphi\|_{\dot{B}^{2}_{2,1}}^{h}\,d\tau+\int_{0}^{t}\|\u\cdot\nabla\varphi\|_{\dot{B}^{2}_{2,1}}^{h}\,d\tau\\
				&+\int_{0}^{t}\sum_{j\geq -1}2^{2j}\|[\dot{\Delta}_j,\u\cdot\nabla]\nabla\varphi\|_{L^2}\,d\tau+\int_{0}^{t}\sum_{j\geq -1}2^{2j}\|[\dot{\Delta}_j,\u\cdot\nabla]\u\|_{L^2}\,d\tau
				\\&+\int_{0}^{t}\|{F_2}\|_{\dot{B}^{2}_{2,1}}^{h}\,d\tau+\int_{0}^{t}\|({F_3},{F_4},{F_5})\|_{\dot{B}^{3}_{2,1}}^{h}\,d\tau.
			\end{split}
		\end{equation}
		To recover the parabolic smoothing effect of $\u$, it follows from $\eqref{frelocal}_2$ and $\eqref{frelocal}_3$ that\begin{equation}\label{similar1}
			\begin{split}
				&\frac{1}{2}\frac{d}{dt}(\|\ddj \u\|^2_{L^2}+\|\ddj \theta\|^2_{L^2})+\|\nabla\ddj \u\|^2_{L^2}+\|\nabla\ddj \theta\|^2_{L^2}\\&\quad\lesssim \|(\nabla\dot{\Delta}_j\varphi,\dot{\Delta}_j{F}_2)\|_{L^2}\|\ddj \u\|_{L^2}+\|\dot{\Delta}_j{F}_3\|_{L^2}\|\ddj \theta\|_{L^2},
			\end{split}
		\end{equation}
		from which and Berstein's inequality, we further have\begin{equation}\label{estil2}
			\begin{split}
				\|({\u},{\ta})\|_{\widetilde{L}^{\infty}_t(\dot{B}^{2}_{2,1})}^{h}+\|({\u},{\ta})\|_{L^1_t(\dot{B}^{4}_{2,1})}^{h}\lesssim \|({\u}_0,{\ta}_0)\|_{\dot{B}^{2}_{2,1}}^{h}+\|\varphi\|_{L_t^1(\dot{B}^{3}_{2,1})}^{h}+\int_{0}^{t}\|({F_2},{F_3})\|_{\dot{B}^{2}_{2,1}}^{h}\,d\tau.
			\end{split}
		\end{equation}Combining \eqref{estil1} and \eqref{estil2} leads to\begin{equation}\label{estil3}
			\begin{split}
				\|\aa&\|_{\widetilde{L}^{\infty}_t(\dot{B}^{3}_{2,1})}^{h}+\|({\u},{\ta})\|_{\widetilde{L}^{\infty}_t(\dot{B}^{2}_{2,1})}^{h}
				+\|\aa\|_{L^1_t(\dot{B}^{3}_{2,1})}^{h}+\|({\u},{\ta})\|_{L^1_t(\dot{B}^{4}_{2,1})}^{h}\\
				&\lesssim\|\aa_0\|_{\dot{B}^{3}_{2,1}}^{h}+\|({\u}_0,{\ta}_0)\|_{\dot{B}^{2}_{2,1}}^{h}+\mathcal{N}(t).
			\end{split}
		\end{equation}Notice the parabolic structure of $\eqref{frelocal}_3$, arguing as \eqref{similar1}-\eqref{estil2}, we also have\begin{equation}\label{estil4}
			\begin{split}
				\|{\ta}\|_{\widetilde{L}^{\infty}_t(\dot{B}^{3}_{2,1})}^{h}+\|{\ta}\|_{L^1_t(\dot{B}^{5}_{2,1})}^{h}\lesssim \|{\ta}_0\|_{\dot{B}^{3}_{2,1}}^{h}+\|\u\|_{L_t^1(\dot{B}^{4}_{2,1})}^{h}+\int_{0}^{t}\|{F_3}\|_{\dot{B}^{3}_{2,1}}^{h}\,d\tau.
			\end{split}
		\end{equation}Thus, we eventually get from the combination of \eqref{estil3} and \eqref{estil4} that\begin{equation}\label{han20}
			\begin{split}
				&\|({\aa},\mathbf{\ta})\|_{\widetilde{L}^{\infty}_t(\dot{B}^{3}_{2,1})}^h+\|\u\|_{\widetilde{L}^{\infty}_t(\dot{B}^
					{2}_{2,1})}^h+\|{\aa}\|_{{L}^{1}_t(\dot{B}^{3}_{2,1})}^h
				+\|\u\|_{{L}^{1}_t(\dot{B}^{4}_{2,1})}^h+\|\mathbf{\ta}\|_{L^1_t(\dot B^{5}_{2,1})}^h\\
				&\quad\lesssim\|(\aa_0,{\ta}_0)\|_{\dot{B}^{3}_{2,1}}^{h}+\|{\u}_0\|_{\dot{B}^{2}_{2,1}}^{h}+\mathcal{N}(t).
			\end{split}
		\end{equation}
		\par
		Next, we turn to the nonlinear estimates for $\mathcal{N}(t)$. By virtue of \eqref{law}, \eqref{jynl} and the embedding $\dot{B}^{1}_{2,1}(\mathbb{R}^2)\hookrightarrow L^\infty(\mathbb{R}^2)$, we have\begin{equation}\label{nestimates}
			\begin{split}
				\|\nabla\u\|_{L^\infty}\|\varphi\|_{\dot{B}^{3}_{2,1}}^h+&\|\nabla\u\cdot\nabla\varphi\|_{\dot{B}^{2}_{2,1}}^{h}\lesssim \|\u\|_{\dot{B}^{2}_{2,1}}\|\varphi\|_{\dot B^{3}_{2,1}}+\|\varphi\|_{\dot B^{2}_{2,1}}\|\u\|_{\dot{B}^{3}_{2,1}}\lesssim \mathcal{E}(\tau)\mathcal{D}(\tau),  \\
				\|(\u\cdot\nabla\varphi,\u\cdot\nabla\u)\|_{\dot B^{2}_{2,1}}^h&\lesssim  \|\u\|_{\dot{B}^{1}_{2,1}}\|(\varphi,\u)\|_{\dot B^{3}_{2,1}}+\|(\varphi,\u)\|_{\dot B^{2}_{2,1}}\|\u\|_{\dot{B}^{2}_{2,1}}\lesssim \mathcal{E}(\tau)\mathcal{D}(\tau),\\
				\|\u\cdot\nabla\theta\|_{\dot B^{3}_{2,1}}^h&\lesssim  \|\u\|_{\dot{B}^{1}_{2,1}}\|\theta\|_{\dot B^{4}_{2,1}}+\|\theta\|_{\dot B^{2}_{2,1}}\|\u\|_{\dot{B}^{3}_{2,1}}\lesssim \mathcal{E}(\tau)\mathcal{D}(\tau),\\
				\|(\varphi\div\u,\theta\div\u)\|_{\dot B^{3}_{2,1}}^h&\lesssim
				\|(\varphi,\theta)\|_{\dot{B}^{1}_{2,1}}\|\u\|_{\dot B^{4}_{2,1}}+\|\u\|_{\dot B^{2}_{2,1}}\|(\varphi,\theta)\|_{\dot{B}^{3}_{2,1}}\lesssim \mathcal{E}(\tau)\mathcal{D}(\tau).
			\end{split}
		\end{equation}
		In addition to \eqref{ag1}, the nonlinear terms involving $I(a)$ can be bounded as\begin{equation}\label{nestimates2}
			\begin{split}
				\|I(a)\Delta\theta\|_{\dot B^{3}_{2,1}}^h
				&\lesssim \|a\|_{\dot{B}^{1}_{2,1}}\|\theta\|_{\dot B^{5}_{2,1}}+\|a\|_{\dot{B}^{3}_{2,1}}\|\theta\|_{\dot B^{3}_{2,1}}\lesssim \mathcal{E}(\tau)\mathcal{D}(\tau),\\
				\|(I(a)\nabla\varphi,I(a)\nabla\theta)\|_{\dot{B}^{2}_{2,1}}^h&\lesssim  \|a\|_{\dot{B}^{1}_{2,1}}\|(\varphi,\theta)\|_{\dot B^{3}_{2,1}}+\|a\|_{\dot{B}^{2}_{2,1}}\|(\varphi,\theta)\|_{\dot B^{2}_{2,1}}\lesssim \mathcal{E}(\tau)\mathcal{D}(\tau),\\
				\|I(a)\Delta\u\|_{\dot{B}^{2}_{2,1}}^h&\lesssim  \|a\|_{\dot{B}^{1}_{2,1}}\|\u\|_{\dot B^{4}_{2,1}}+\|a\|_{\dot{B}^{2}_{2,1}}\|\u\|_{\dot B^{3}_{2,1}}\lesssim \mathcal{E}(\tau)\mathcal{D}(\tau),\\
				\|I(a)(2|D(\u)|^2&-(\div \u)^2)\|_{\dot{B}^{3}_{2,1}}^{h}+\Big\|\frac{1}{1+a}(2|D(\u)|^2-(\div \u)^2)\Big\|_{\dot{B}^{3}_{2,1}}^{h}\\
				&\lesssim  \|I(a)\|_{\dot{B}^{3}_{2,1}}\|D\u\otimes D\u\|_{\dot{B}^{1}_{2,1}}+(1+\|I(a)\|_{\dot{B}^{1}_{2,1}})\|D\u\otimes D\u\|_{\dot{B}^{3}_{2,1}}\\
				&\lesssim \|a\|_{\dot{B}^{3}_{2,1}}\|\u\|_{\dot{B}^{2}_{2,1}}\|\u\|_{\dot{B}^{2}_{2,1}}+(1+\|a\|_{\dot{B}^{1}_{2,1}})\|\u\|_{\dot{B}^{2}_{2,1}}\|\u\|_{\dot{B}^{4}_{2,1}}\\
				&\lesssim (1+\mathcal{E}(\tau))\mathcal{E}(\tau)\mathcal{D}(\tau).
			\end{split}
		\end{equation}
		Moreover, making use of \eqref{jiaohuanzi1}, the commutator terms are bounded as\begin{equation}\label{nestimates3}
			\begin{split}
				\sum_{j\geq -1}2^{2j}\|[\dot{\Delta}_j,\u\cdot\nabla]\nabla\varphi\|_{L^2}+\sum_{j\geq -1}2^{2j}\|[\dot{\Delta}_j,\u\cdot\nabla]\u\|_{L^2}&\lesssim \|\u\|_{\dot{B}^{2}_{2,1}}(\|\varphi\|_{\dot{B}^{3}_{2,1}}+\|\u\|_{\dot{B}^{2}_{2,1}})\\
				&\lesssim \mathcal{E}(\tau)\mathcal{D}(\tau).
			\end{split}
		\end{equation}
		Inserting \eqref{nestimates}, \eqref{nestimates2} and \eqref{nestimates3} into \eqref{han20}, we end up with the desired estimate \eqref{han201}.
	\end{proof}
	
	\subsection{Estimates of $(a,b)$}To enclose the a priori estimates \eqref{han200} and \eqref{han201}, we need to derive the following $L^\infty$ estimates of the non-dissipative variables $(a,b)$ in time.
	\begin{lemma}\label{ablemma}
		For any $t>0$, it holds that \begin{equation}\label{han202}
			\begin{split}
				\|(a,b)\|_{\widetilde{L}^\infty_t(\dot B^{0}_{2,1})}^\ell+\|(a,b)\|_{\widetilde{L}^\infty_t(\dot B^{3}_{2,1})}^h
				&\lesssim \|(a_0,b_0,\aa_0,{\u}_0,{\ta}_0)\|_{\dot{B}^{0}_{2,1}}^{\ell} +\|(a_0,b_0,\varphi_0,\mathbf{\ta}_0)\|_{\dot{B}^{3}_{2,1}}^h\\
				&\quad+\|{\u}_{0}\|_{\dot{B}^{2}_{2,1}}^h+\int_{0}^{t}(1+\mathcal{E}(\tau))\mathcal{E}(\tau)\mathcal{D}(\tau)\,d\tau.
			\end{split}
		\end{equation}
	\end{lemma}
	\begin{proof}
		Set\begin{equation}
			\begin{split}
				\delta:=\varphi-2a,
			\end{split}
		\end{equation}then it follows from $\eqref{mm3}_1$ and $\eqref{linequ}_1$ that\begin{equation}\label{deltaeq}
			\begin{split}
				\partial_t\delta+\u\cdot\nabla\delta=2a\div\u+{F}_4,
			\end{split}
		\end{equation}where ${F}_4$ was defined in \eqref{jihao}. Applying $\dot{\Delta}_j$ to \eqref{deltaeq} and using a commutator argument, we have\begin{equation}
			\begin{split}
				\partial_t\dot{\Delta}_j\delta+\u\cdot\nabla\dot{\Delta}_j\delta=-[\dot{\Delta}_j,\u\cdot\nabla]\delta+2\dot{\Delta}_j(a\div\u)+\dot{\Delta}_j{F}_4.
			\end{split}
		\end{equation}Standard $L^2$ estimates for the transport equation gives
		\begin{equation}
			\begin{split}
				\|\dot{\Delta}_j\delta(t)\|_{L^2}\lesssim& \|\dot{\Delta}_j\delta_0\|_{L^2}+\int_{0}^{t}\|\div\u\|_{L^\infty}\|\dot{\Delta}_j\delta\|_{L^2}\,d\tau+\int_{0}^{t}\|[\dot{\Delta}_j,\u\cdot\nabla]\delta\|_{L^2}\,d\tau\\&+\int_{0}^{t}\|\dot{\Delta}_j(a\div\u)\|_{L^2}\,d\tau+\int_{0}^{t}\|\dot{\Delta}_j{F}_4\|_{L^2}\,d\tau,
			\end{split}
		\end{equation}
		which further leads to
		\begin{equation}\label{jjiayou}
			\begin{split}
				\|\delta\|_{\widetilde{L}^\infty_t(\dot B^{0}_{2,1})}^\ell+\|\delta\|_{\widetilde{L}^\infty_t(\dot B^{3}_{2,1})}^h
				\lesssim& \|\delta_0\|_{\dot B^{0}_{2,1}}^\ell+\|\delta_0\|_{\dot B^{3}_{2,1}}^h+\int_{0}^{t}\|{F}_4\|_{\dot B^{0}_{2,1}}^\ell\,d\tau+\int_{0}^{t}\|{F}_4\|_{\dot B^{3}_{2,1}}^h\,d\tau
				\\&+\int_{0}^{t}\|a\div\u\|_{\dot B^{0}_{2,1}}^\ell\,d\tau+\int_{0}^{t}\|a\div\u\|_{\dot B^{3}_{2,1}}^h\,d\tau\\&+\int_0^t\sum_{j\le 0}\|[\ddj,\u\cdot\nabla]{\delta}\|_{L^2}\,d\tau+\int_0^t\sum_{j\ge -1}2^{3j}\|[\ddj,\u\cdot\nabla]{\delta}\|_{L^2}\,d\tau\\&+\int_0^t\|\div \u\|_{L^\infty}\|\delta\|_{\dot B^{0}_{2,1}}^\ell\,d\tau+\int_0^t\|\div \u\|_{L^\infty}\|\delta\|_{\dot B^{3}_{2,1}}^h\,d\tau.
			\end{split}
		\end{equation}
		It was shown in the previous two subsections that\begin{equation}
			\begin{split}
				\|{F}_4\|_{\dot B^{0}_{2,1}}^\ell+\|{F}_4\|_{\dot B^{3}_{2,1}}^h\lesssim (1+\mathcal{E}(\tau))\mathcal{E}(\tau)\mathcal{D}(\tau),
			\end{split}
		\end{equation}so we only need to deal with the last three lines in \eqref{jjiayou}. By means of \eqref{law}, \eqref{daishu} and \eqref{jynl}, we can tackle the second line as\begin{equation}
			\begin{split}
				\|a\div\u\|_{\dot B^{0}_{2,1}}^\ell&\lesssim \|a\|_{\dot{B}^{0}_{2,1}}\|\u\|_{\dot{B}^{2}_{2,1}}\lesssim \mathcal{E}(\tau)\mathcal{D}(\tau),\\
				\|a\div\u\|_{\dot B^{3}_{2,1}}^h&\lesssim \|a\|_{\dot{B}^{1}_{2,1}}\|\u\|_{\dot{B}^{4}_{2,1}}+\|a\|_{\dot{B}^{3}_{2,1}}\|\u\|_{\dot{B}^{2}_{2,1}}\lesssim \mathcal{E}(\tau)\mathcal{D}(\tau).
			\end{split}
		\end{equation}Taking advantage of \eqref{jiaohuanzi1} and \eqref{jiaohuanzi2}, the third line can be bounded as\begin{equation}
			\begin{split}
				\sum_{j\le 0}\|[\ddj,\u\cdot\nabla]{\delta}\|_{L^2}
				&\lesssim \|\u\|_{\dot B^{2}_{2,1}}\|\delta\|_{\dot B^{0}_{2,1}}\lesssim \|\u\|_{\dot{B}^{2}_{2,1}}\|(a,\varphi)\|_{\dot{B}^{0}_{2,1}}\\
				&\lesssim \mathcal{D}(\tau)\mathcal{E}(\tau),\\
				\sum_{j\ge -1}2^{3j}\|[\ddj,\u\cdot\nabla]{\delta}\|_{L^2}
				&\lesssim \|\nabla\u\|_{L^\infty}\|\delta\|_{\dot B^{3}_{2,1}}+\|\u\|_{\dot B^{3}_{2,1}}\|\nabla\delta\|_{L^\infty}
				\\&\lesssim \|\u\|_{\dot{B}^{2}_{2,1}}\|(a,\varphi)\|_{\dot{B}^{3}_{2,1}}+\|\u\|_{\dot{B}^{3}_{2,1}}\|(a,\varphi)\|_{\dot{B}^{2}_{2,1}}\\
				&\lesssim \mathcal{D}(\tau)\mathcal{E}(\tau).
			\end{split}
		\end{equation}Besides, it is clear that the last line is bounded by $\mathcal{D}(\tau)\mathcal{E}(\tau)$. Thus, substituting all nonlinear estimates above into \eqref{jjiayou}, we arrive at\begin{equation}\label{deltaes}
			\begin{split}
				&\|\delta\|_{\widetilde{L}^\infty_t(\dot B^{0}_{2,1})}^\ell+\|\delta\|_{\widetilde{L}^\infty_t(\dot B^{3}_{2,1})}^h
				\lesssim \|(a_0,\varphi_0)\|_{\dot B^{0}_{2,1}}^\ell+\|(a_0,\varphi_0)\|_{\dot B^{3}_{2,1}}^h+\int_{0}^{t}(1+\mathcal{E}(\tau))\mathcal{E}(\tau)\mathcal{D}(\tau)\,d\tau,
			\end{split}
		\end{equation}Owing to the fact that $a=\frac{1}{2}(\varphi-\delta)$, we finally obtain from \eqref{han200}, \eqref{han201} and \eqref{deltaes} that\begin{equation}\label{adgj}
			\begin{split}
				\|a\|_{\widetilde{L}^\infty_t(\dot B^{0}_{2,1})}^\ell+\|a\|_{\widetilde{L}^\infty_t(\dot B^{3}_{2,1})}^h
				&\lesssim \|(\varphi,\delta)\|_{\widetilde{L}^\infty_t(\dot B^{0}_{2,1})}^\ell+\|(\varphi,\delta)\|_{\widetilde{L}^\infty_t(\dot B^{3}_{2,1})}^h
				\\&\lesssim \|(a_0,\aa_0,{\u}_0,{\ta}_0)\|_{\dot{B}^{0}_{2,1}}^{\ell} +\|(a_0,\varphi_0,\mathbf{\ta}_0)\|_{\dot{B}^{3}_{2,1}}^h+\|{\u}_{0}\|_{\dot{B}^{2}_{2,1}}^h\\
				&\quad+\int_{0}^{t}(1+\mathcal{E}(\tau))\mathcal{E}(\tau)\mathcal{D}(\tau)\,d\tau.
			\end{split}
		\end{equation}
		In the same manner, we also can infer from $\eqref{mm3}_4$ that\begin{equation}\label{bdgj}
			\begin{split}
				\|b\|_{\widetilde{L}^\infty_t(\dot B^{0}_{2,1})}^\ell+\|b\|_{\widetilde{L}^\infty_t(\dot B^{3}_{2,1})}^h
				&\lesssim \|(b_0,\aa_0,{\u}_0,{\ta}_0)\|_{\dot{B}^{0}_{2,1}}^{\ell} +\|(b_0,\varphi_0,\mathbf{\ta}_0)\|_{\dot{B}^{3}_{2,1}}^h+\|{\u}_{0}\|_{\dot{B}^{2}_{2,1}}^h\\
				&\quad+\int_{0}^{t}(1+\mathcal{E}(\tau))\mathcal{E}(\tau)\mathcal{D}(\tau)\,d\tau.
			\end{split}
		\end{equation}
		The proof of Lemma \ref{ablemma} is completed.
	\end{proof}

	\subsection{Completing the proof of Theorem \ref{dingli2}}
	Denote $T^*$ by the life-span of the local solution $(a,\u,\theta,b)$ constructed in Proposition \ref{dingli1}. Hence, to prove Theorem \ref{dingli2}, it suffices to prove that $T^*=\infty$ and there holds \eqref{xiaonorm}.\par  Denote \begin{equation}\label{normdef2}
		\begin{split}
			\mathcal{X}(t):=&		\norm{(\varphi,\u,{\ta},a,{b})}_{\widetilde{L}^{\infty}_t(\dot{B}^{0}_{2,1})}^{\ell}+\norm{(\varphi,{\ta},a,{b})}_{\widetilde{L}^{\infty}_t(\dot{B}^{3}_{2,1})}^{h}
			+\norm{\u}_{\widetilde{L}^{\infty}_t(\dot{B}^{2}_{2,1})}^h\\
			&+\norm{(\aa,\u,{\ta})}_{L^1_t(\dot{B}^{2}_{2,1})}^{\ell}
			+\norm{\aa}_{L^{1}_t(\dot{B}^{2}_{2,1})}^h
			+\norm{\u}_{L^1_t(\dot{B}^{4}_{2,1})}^h+\norm{\ta}_{L^1_t(\dot{B}^{5}_{2,1})}^h.
		\end{split}
	\end{equation}
	Recall
	$\varphi_0=a_0+a_0\theta_0+\frac{1}{2}b_0^2+b_0$, it follows from \eqref{law} and \eqref{daishu} that
	\begin{equation}\label{tadi}
		\begin{split}
			\norm{\varphi_0}_{\dot{B}^{0}_{2,1}}^{\ell}+\norm{\varphi_0}_{\dot{B}^{3}_{2,1}}^{h}
			&\lesssim\norm{(a_0,b_0)}_{\dot{B}^{0}_{2,1}}^{\ell}+\norm{(a_0,b_0)}
			_{\dot{B}^{3}_{2,1}}^{h}+\norm{a_0}_{\dot{B}^{0}_{2,1}}\norm{\theta_0}_{\dot{B}^{1}_{2,1}}+\norm{a_0}_{\dot{B}^{1}_{2,1}}\norm{\theta_0}_{\dot{B}^{3}_{2,1}}\\
			&\qquad+\norm{a_0}_{\dot{B}^{3}_{2,1}}\norm{\theta_0}_{\dot{B}^{1}_{2,1}}
			+\norm{b_0}_{\dot{B}^{0}_{2,1}}\norm{b_0}_{\dot{B}^{1}_{2,1}}+\norm{b_0}_{\dot{B}^{1}_{2,1}}\norm{b_0}_{\dot{B}^{3}_{2,1}}\\
			&\leq   C_1(1+\mathcal{X}_0)\mathcal{X}_0
		\end{split}
	\end{equation}for some constant $C_1>1$, where $\mathcal{X}_0$ is defined in \eqref{smallness}.
	Therefore, taking $\varepsilon\in(0,1)$ in \eqref{smallness} sufficiently small enables us to define\begin{equation}
		\begin{split}
			T^{**}:=\sup\Big\{t\in [0,T^*): \mathcal{X}(t)\leq 4C_1\mathcal{X}_0\quad\text{and}\quad\text{\eqref{eq:smalla} holds }\Big\}.
		\end{split}
	\end{equation}
	Then, adding up \eqref{han200}, \eqref{han201} and \eqref{han202} yields that for any $t\in [0,T^{**})$,\begin{equation}\label{close}
		\begin{split}
			\mathcal{X}(t)&\lesssim (1+\mathcal{X}_0)\mathcal{X}_0+\int_{0}^{t}(1+\mathcal{E}(\tau))\mathcal{E}(\tau)\mathcal{D}(\tau)\,d\tau\\&\lesssim (1+\mathcal{X}_0)\mathcal{X}_0+\Big(1+\sup_{\tau\in [0,t]}\mathcal{E}(\tau)\Big)\cdot\sup_{\tau\in [0,t]}\mathcal{E}(\tau)\cdot\int_{0}^{t}\mathcal{D}(\tau)\,d\tau\\
			&\lesssim (1+\mathcal{X}_0)\mathcal{X}_0+(1+\mathcal{X}(t))\mathcal{X}^2(t).
		\end{split}
	\end{equation}
	Based on the estimate above, it is easy to use a standard continuity argument to prove that $T^{**} =\infty$ provided that $\varepsilon\ll 1$. The details are omitted here. This completes the proof of Theorem \ref{dingli2}.
	
	\section{The proof of Theorem \ref{dingli3}}\label{sec3}
	This section is devoted to deriving the optimal time-decay rates for the dissipative variables $(\varphi,\u,\theta)$. Inspired by \cite{MR3005540,MR4188989}, we first show that the negative Besov norms of $(a,\varphi,\u,\theta)$ at low frequencies are bounded along time evolution. Then, by virtue of the interpolation inequality, we establish a Lyapunov-type differential inequality for the energy norm of $(\varphi,\u,\theta)$, which leads to the desired time-decay estimates. We remark that the notations in \eqref{normdef} continue to be used in this section.
	
	\subsection{The boundness of negative Besov norm}
	\begin{proposition}\label{propagate}
		Let $0<\sigma\le1$ and $\|(a_0,\u_0,\ta_0,b_0)\|^{\ell}_{\dot B^{{-\sigma}}_{2,\infty}}<\infty$, then it holds that for all $t\geq 0$,\begin{equation}\label{sa39-1}
			\begin{split}
				\mathcal{Y}(t):=\|(a,\aa,\u,\ta)(t)\|^{\ell}_{\dot B^{{-\sigma}}_{2,\infty}}\leq C_0,
			\end{split}
		\end{equation}
		where the constant $C_0>0$ depends on the norms of the initial data.
	\end{proposition}
	\begin{proof}
		It follows from \eqref{qi4} that\begin{equation}
			\begin{split}
				\frac{d}{dt}&\|(\ddj \varphi,\ddj \u,\ddj \theta)\|^2_{L^2} +\|(\nabla\ddj \u,\nabla\ddj \theta)\|^2_{L^2}\\&\lesssim \|(\ddj {F}_1,\ddj {F}_2,\ddj {F}_3)\|_{L^2}\|(\ddj \varphi,\ddj \u,\ddj \theta)\|_{L^2}.
			\end{split}
		\end{equation}
		By performing a routine procedure, we obtain\begin{equation}\label{sa9}
			\begin{split}
				\norm{(\aa,\u,\ta)(t)}^{\ell}_{\dot B^{{-\sigma}}_{2,\infty}}
				\lesssim\norm{(\aa_0,\u_0,\ta_0)}^{\ell}_{\dot{B}^{{-\sigma}}_{2,\infty}}
				+\int_{0}^{t}\norm{({F_1},{F_2},{F_3})}_{\dot{B}^{{-\sigma}}_{2,\infty}}^{\ell}\,d\tau.
			\end{split}
		\end{equation}
		To control $\norm{a(t)}^{\ell}_{\dot{B}^{{-\sigma}}_{2,\infty}}$, we revert to \eqref{deltaeq} and get by a similar derivation of \eqref{jjiayou} that
		\begin{equation}\label{newdelta}
			\begin{split}
				\|\delta(t)\|_{\dot B^{{-\sigma}}_{2,\infty}}^\ell
				\lesssim &\|\delta_0\|_{\dot B^{-\sigma}_{2,\infty}}^\ell+\int_{0}^{t}\|{F}_4\|_{\dot{B}^{{-\sigma}}_{2,\infty}}^{\ell}\,d\tau+\int_{0}^{t}\|a\div\u\|_{\dot{B}^{{-\sigma}}_{2,\infty}}^{\ell}\,d\tau\\&+\int_0^t\sup_{j\leq 0}2^{-j\sigma}\|[\ddj,\u\cdot\nabla]{\delta}\|_{L^2}\,d\tau+\int_0^t\|\div \u\|_{L^\infty}\|\delta\|_{\dot B^{-\sigma}_{2,\infty}}^\ell\,d\tau,
			\end{split}
		\end{equation}which together with \eqref{sa9} yields\begin{equation}\label{sa12}
			\begin{split}
				\mathcal{Y}(t)
				\lesssim&\norm{(a_0,\varphi_0,\u_0,\ta_0)}^{\ell}_{\dot{B}^{{-\sigma}}_{2,\infty}}
				+\int_{0}^{t}\|({F}_1,{F}_2,{F}_3,{F}_4)\|_{\dot{B}^{{-\sigma}}_{2,\infty}}^{\ell}\,d\tau+\int_{0}^{t}\|a\div\u\|_{\dot{B}^{{-\sigma}}_{2,\infty}}^{\ell}\,d\tau\\&+\int_0^t\sup_{j\leq 0}2^{-j\sigma}\|[\ddj,\u\cdot\nabla]{\delta}\|_{L^2}\,d\tau+\int_{0}^{t}\mathcal{D}(\tau)\mathcal{Y}(\tau)\,d\tau.
			\end{split}
		\end{equation}To estimate the nonlinear terms in the right-hand side above, we need the following product estimates stem from \eqref{guangyidaishu} with $(p,d,s_1,s_2)=(2,2,1,-\sigma)$: \begin{equation}\label{key0}
			\begin{split}
				\norm{fg}_{\dot{B}_{2,\infty}^{{-\sigma}}}\lesssim \norm{f}_{\dot{B}_{2,\infty}^{{-\sigma}}}\norm{g}_{\dot{B}_{2,1}^{1}},\quad 0<\sigma\leq 1.
			\end{split}
		\end{equation}
		Besides, we observe that\begin{equation}\label{fenjie123}
			\begin{split}
				\|(a,\aa,\u,\ta)\|_{\dot B^{{-\sigma}}_{2,\infty}}&\lesssim \|(a,\aa,\u,\ta)\|^{\ell}_{\dot B^{{-\sigma}}_{2,\infty}}+\|(a,\aa,\ta)\|^{h}_{\dot B^{{3}}_{2,\infty}}+\|\u\|^{h}_{\dot B^{{2}}_{2,\infty}}\\
				&\lesssim \mathcal{Y}(\tau)+\mathcal{E}(\tau).
			\end{split}
		\end{equation}Hence, resorting to \eqref{key0}, \eqref{fenjie123} and \eqref{jynl}, we get\begin{equation}\label{sa33}
			\begin{split}
				\|(\u\cdot\nabla\varphi,\u\cdot\nabla\u,\u\cdot\nabla\ta)\|_{\dot{B}^{{-\sigma}}_{2,\infty}}^{\ell}&\lesssim \|\u\|_{\dot B^{-\sigma}_{2,\infty}}\|(\varphi,\u,\ta)\|_{\dot B^{2}_{2,1}}\lesssim (\mathcal{Y}(\tau)+\mathcal{E}(\tau))\mathcal{D}(\tau),
				\\
				\|(\varphi\div\u,\ta\div\u,a\div\u)\|_{\dot{B}^{{-\sigma}}_{2,\infty}}^{\ell}&\lesssim \|(\varphi,\ta,a)\|_{\dot{B}^{{-\sigma}}_{2,\infty}}\|\u\|_{\dot{B}^{2}_{2,1}}\lesssim (\mathcal{Y}(\tau)+\mathcal{E}(\tau))\mathcal{D}(\tau).
			\end{split}
		\end{equation}Keeping in mind that $I(a)=a-aI(a)$, then using \eqref{key0}, \eqref{ag1}, \eqref{fenjie123} and \eqref{jynl} gives\begin{equation}\label{du17}
			\begin{split}
				\|I(a)\|_{\dot B^{-\sigma}_{2,\infty}}&\lesssim \|a\|_{\dot B^{-\sigma}_{2,\infty}}+\|a\|_{\dot B^{-\sigma}_{2,\infty}}\|I(a)\|_{\dot B^{1}_{2,1}}\\&\lesssim \|a\|_{\dot B^{-\sigma}_{2,\infty}}(1+\|a\|_{\dot B^{1}_{2,1}})\\
				&\lesssim (\mathcal{Y}(\tau)+\mathcal{E}(\tau))(1+\mathcal{E}(\tau)).
			\end{split}
		\end{equation}As a result, applying \eqref{key0} and \eqref{daishu} to nonlinear terms involving $I(a)$ implies\begin{equation}\label{zhjy1}
			\begin{split}
				\|(I(a)\Delta\theta,I(a)\Delta\u)\|_{\dot{B}^{{-\sigma}}_{2,\infty}}^{\ell}&\lesssim \|I(a)\|_{\dot B^{-\sigma}_{2,\infty}}\|(\theta,\u)\|_{\dot B^{3}_{2,1}}\lesssim (\mathcal{Y}(\tau)+\mathcal{E}(\tau))(1+\mathcal{E}(\tau))\mathcal{D}(\tau),\\
				\|(I(a)\nabla\varphi,I(a)\nabla\theta)\|_{\dot{B}^{{-\sigma}}_{2,\infty}}^{\ell}&\lesssim \|I(a)\|_{\dot B^{-\sigma}_{2,\infty}} \|(\varphi,\theta)\|_{\dot B^{2}_{2,1}}\lesssim  (\mathcal{Y}(\tau)+\mathcal{E}(\tau))(1+\mathcal{E}(\tau))\mathcal{D}(\tau),\\
				\|I(a)(2|D(\u)|^2&-(\div \u)^2)\|_{\dot{B}^{{-\sigma}}_{2,\infty}}^{\ell}+\Big\|\frac{1}{1+a}(2|D(\u)|^2-(\div \u)^2)\Big\|_{\dot{B}^{{-\sigma}}_{2,\infty}}^{\ell}\\&\lesssim \|I(a)\|_{\dot B^{-\sigma}_{2,\infty}}\|D\u\otimes D\u\|_{\dot B^{1}_{2,1}}+\|D\u\otimes D\u\|_{\dot B^{-\sigma}_{2,\infty}}\\
				&\lesssim \|I(a)\|_{\dot B^{-\sigma}_{2,\infty}}\|\u\|_{\dot B^{2}_{2,1}}\|\u\|_{\dot B^{2}_{2,1}}+\|\u\|_{\dot B^{1-\sigma}_{2,1}}\|\u\|_{\dot B^{2}_{2,1}}\\
				&\lesssim (\mathcal{Y}(\tau)+\mathcal{E}(\tau))(1+\mathcal{E}(\tau))\mathcal{E}(\tau)\mathcal{D}(\tau)+\mathcal{E}(\tau)\mathcal{D}(\tau).
			\end{split}
		\end{equation}
		Thanks to \eqref{jiaohuanzi3}, we also have\begin{equation}\label{sa922}
			\begin{split}
				\sup_{j\leq 0}2^{-j\sigma}\|[\ddj,\u\cdot\nabla]{\delta}\|_{L^2}
				&\lesssim \|\u\|_{\dot{B}^{2}_{2,1}}\|\delta\|_{\dot B^{{-\sigma}}_{2,\infty}}\\
				&\lesssim \|\u\|_{\dot{B}^{2}_{2,1}}(\|(a,\varphi)\|_{\dot B^{{-\sigma}}_{2,\infty}}^\ell+\|(a,\varphi)\|_{\dot B^{3}_{2,1}}^h)\\
				&\lesssim \mathcal{D}(\tau)(\mathcal{Y}(\tau)+\mathcal{E}(\tau)).
			\end{split}
		\end{equation}Inserting \eqref{sa33}, \eqref{zhjy1} and \eqref{sa922} into \eqref{sa12}, we finally arrive at\begin{equation}\label{sa36}
			\begin{split}
				\mathcal{Y}(t)
				\lesssim&\norm{(a_0,\varphi_0,\u_0,\ta_0)}^{\ell}_{\dot{B}^{{-\sigma}}_{2,\infty}}
				+\int_{0}^{t}(1+\mathcal{E}(\tau))^2\mathcal{E}(\tau)\mathcal{D}(\tau)\,d\tau\\&+\int_{0}^{t}(1+\mathcal{E}(\tau))^2\mathcal{D}(\tau)\mathcal{Y}(\tau)\,d\tau.
			\end{split}
		\end{equation}
		According to the definition of $\varphi_0$ in \eqref{linequ}, using \eqref{key0} implies\begin{equation}
			\begin{split}
				\norm{\varphi_0}^{\ell}_{\dot{B}^{{-\sigma}}_{2,\infty}}
				\lesssim&\norm{(a_0,b_0)}^{\ell}_{\dot{B}^{{-\sigma}}_{2,\infty}}+\norm{a_0}_{\dot{B}^{{-\sigma}}_{2,\infty}}
				\norm{\ta_0}_{\dot{B}^{{1}}_{2,1}}+\norm{b_0}_{\dot{B}^{{-\sigma}}_{2,\infty}}
				\norm{b_0}_{\dot{B}^{{1}}_{2,1}}\\
				\lesssim&\norm{(a_0,b_0)}^{\ell}_{\dot{B}^{{-\sigma}}_{2,\infty}}+(\norm{(a_0,b_0)}^{\ell}_{\dot{B}^{{-\sigma}}_{2,\infty}}+\norm{(a_0,b_0)}^{h}_{\dot{B}^{3}_{2,1}})(\norm{(\theta_0,b_0)}^{\ell}_{\dot{B}^{0}_{2,1}}\\&+\norm{(\theta_0,b_0)}^{h}_{\dot{B}^{3}_{2,1}}).
			\end{split}
		\end{equation}
		Keep in mind that for all $t\geq 0$,\begin{equation}
			\begin{split}
				\sup_{\tau\in [0,t]}\mathcal{E}(\tau)+\int_{0}^{t}\mathcal{D}(\tau)\,d\tau\leq \mathcal{X}(t)\lesssim \mathcal{X}_0.
			\end{split}
		\end{equation}
		Consequently, employing Gronwall's inequality to \eqref{sa36} gives us that for all $t\geq 0$,\begin{equation}
			\begin{split}
				\mathcal{Y}(t) \leq C_0,
			\end{split}
		\end{equation}where the constant $C_0>0$ depends on the norms of the initial data. This completes the proof of Proposition \ref{propagate}.
	\end{proof}
	
	\subsection{Completing the proof of Theorem \ref{dingli3}}
	From the proofs of Lemma \ref{dipinlemma} and Lemma \ref{gaopinlemma} in the previous section, we can get the following inequality:\begin{equation}\label{sa1}
		\begin{split}
			&\frac{d}{dt}	(\|(\aa,\u,\ta)\|^\ell_{\dot{B}_{2,1}^{0}}+\|(\aa,\ta)\|^h_{\dot{B}_{2,1}^3}+\|\u\|^h_{\dot{B}_{2,1}^2})+\mathcal{D}(t)\lesssim (1+\mathcal{E}(t))\mathcal{E}(t)\mathcal{D}(t).
		\end{split}
	\end{equation}Since $\mathcal{E}(t)\lesssim \varepsilon\ll 1$ for all $t\geq 0$, we conclude that\begin{equation}\label{sa3}
		\begin{split}
			&\frac{d}{dt}(\|(\aa,\u,\ta)\|^\ell_{\dot{B}_{2,1}^{0}}+\|(\aa,\ta)\|^h_{\dot{B}_{2,1}^3}+\|\u\|^h_{\dot{B}_{2,1}^2})+c\mathcal{D}(t)\leq 0
		\end{split}
	\end{equation}for some constant $c>0$. Thanks to $0<{\sigma}\leq 1$,
	we infer from the interpolation inequality \eqref{inter2} that\begin{equation}
		\begin{split}
			\norm{(\aa,\u,\ta)}^\ell_{\dot{B}_{2,1}^{0}}
			\lesssim \Big(\norm{(\aa,\u,\ta)}^\ell_{\dot{B}_{2,\infty}^{{-\sigma}}}\Big)^{\frac{2}{2+\sigma}}\Big(\norm{(\aa,\u,\ta)}^\ell_{\dot{B}_{2,1}^{2}}\Big)^{\frac{\sigma}{2+\sigma}},
		\end{split}
	\end{equation}
	which along with \eqref{sa39-1} implies that\begin{equation}\label{sa51}
		\begin{split}
			\norm{(\aa,\u,\ta)}^\ell_{\dot{B}_{2,1}^{2}}\geq C\big(\norm{(\aa,\u,\ta)}^\ell_{\dot{B}_{2,1}^{0}}\big)^{1+\frac{\sigma}{2}}.
		\end{split}
	\end{equation}Moreover, it follows from the fact  $\|(\aa,\ta)\|^h_{\dot{B}_{2,1}^3}+\|\u\|^h_{\dot{B}_{2,1}^2}\leq \mathcal{E}(t)\ll 1$ for all $t\geq 0$ that\begin{equation}
		\begin{split}
			\|\aa\|^h_{\dot{B}_{2,1}^3}\geq C \Big(\|\aa\|^h_{\dot{B}_{2,1}^3}\Big)^{1+\frac{\sigma}{2}},\quad  \|\u\|^h_{\dot{B}_{2,1}^4}\geq C \Big(\|\u\|^h_{\dot{B}_{2,1}^2}\Big)^{1+\frac{\sigma}{2}}, \quad \|\ta\|^h_{\dot{B}_{2,1}^5}\geq C \Big(\|\ta\|^h_{\dot{B}_{2,1}^3}\Big)^{1+\frac{\sigma}{2}}.
		\end{split}
	\end{equation}
	Therefore, there exists a constant $\tilde{c}>0$ such that the following  Lyapunov-type inequality holds
	\begin{equation}\label{sa58}
		\begin{split}
			\frac{d}{dt}&(\|(\aa,\u,\ta)\|^\ell_{\dot{B}_{2,1}^{0}}+\|(\aa,\ta)\|^h_{\dot{B}_{2,1}^3}+\|\u\|^h_{\dot{B}_{2,1}^2})\\
			&+\tilde{c} (\|(\aa,\u,\ta)\|^\ell_{\dot{B}_{2,1}^{0}}+\|(\aa,\ta)\|^h_{\dot{B}_{2,1}^3}+\|\u\|^h_{\dot{B}_{2,1}^2})^{1+\frac{\sigma}{2}}\le0.
		\end{split}
	\end{equation}
	Solving this differential inequality directly leads to that for all $t\geq 0$,\begin{equation}\label{sa59}
		\begin{split}
			\|(\aa,\u,\ta)\|^\ell_{\dot{B}_{2,1}^{0}}
			+\|(\aa,\ta)\|^h_{\dot{B}_{2,1}^3}+\|\u\|^h_{\dot{B}_{2,1}^2} \lesssim (1+t)^{-\frac{{\sigma}}{2}},
		\end{split}
	\end{equation}
	from which we get\begin{equation}
		\begin{split}
			\|(\aa,\u,\ta)\|_{\dot{B}_{2,1}^{0}}\lesssim \|(\aa,\u,\ta)\|^\ell_{\dot{B}_{2,1}^{0}}
			+\|(\aa,\ta)\|^h_{\dot{B}_{2,1}^3}+\|\u\|^h_{\dot{B}_{2,1}^2}\lesssim (1+t)^{-\frac{{\sigma}}{2}}.
		\end{split}
	\end{equation}
	Furthermore, resorting to \eqref{inter2} and \eqref{sa39-1} once again yields that for any $\gamma\in (-\sigma,0)$,\begin{equation}
		\begin{split}
			\norm{(\aa,\u,\ta)}_{\dot{B}_{2,1}^{{\gamma}}}&\lesssim \Big(\norm{(\aa,\u,\ta)}_{\dot{B}_{2,\infty}^{{-\sigma}}}\Big)^{-\frac{\gamma}{\sigma}}\Big(\norm{(\aa,\u,\ta)}_{\dot{B}_{2,1}^{0}}\Big)^{1+\frac{\gamma}{\sigma}}\\
			&\lesssim \Big(\norm{(\aa,\u,\ta)}_{\dot{B}_{2,\infty}^{{-\sigma}}}^\ell+\|(\aa,\ta)\|^h_{\dot{B}_{2,1}^3}+\|\u\|^h_{\dot{B}_{2,1}^2}\Big)^{-\frac{\gamma}{\sigma}}(1+t)^{-\frac{{\sigma+\gamma}}{2}}\\
			&\lesssim (1+t)^{-\frac{{\sigma+\gamma}}{2}}.
		\end{split}
	\end{equation}
	Thus, after using the embedding $\dot{B}^0_{2,1}(\mathbb{R}^2)\hookrightarrow L^2(\mathbb{R}^2)$, the proof of Theorem \ref{dingli3} is finished.

	\appendix
	\section{Littlewood-Paley theory}
	For convenience of readers, we here briefly introduce the Littlewood-Paley decomposition, Besov spaces and some related analysis tools. More details and proofs can be found in \cite[Chapter 2]{MR2768550}.\par
	Choose a radial non-increasing function $\chi\in\mathcal{S}(\mathbb{R}^d)$ supported in $\{\xi\in\mathbb{R}^d:|\xi|\leq\frac{4}{3}\}$ and such that $\chi\equiv 1$ in $\{\xi\in\mathbb{R}^d:|\xi|\leq\frac{3}{4}\}$. Then $\varphi(\xi):=\chi(\frac{\xi}{2})-\chi(\xi)$ is supported in $\{\xi\in\mathbb{R}^d:\frac{3}{4}\leq|\xi|\leq\frac{8}{3}\}$ and satisfies\begin{equation*}
		\begin{split}
			\sum_{j\in\mathbb{Z}}\varphi(2^{-j}\xi)&=1\,\,\,\,\text{in \,\,}\mathbb{R}^d \backslash \{0\}.
		\end{split}
	\end{equation*}For any $j\in\mathbb{Z}$, the homogeneous dyadic block $\dot{\Delta}_j$ is defined by\begin{equation*}
		\begin{split}
			\dot{\Delta}_j u&:=\varphi(2^{-j}D)u=2^{2j}\int_{\mathbb{R}^d}g(2^jy)u(x-y)dy\quad\text{with}\quad g:=\mathcal{F}^{-1}\varphi.
		\end{split}
	\end{equation*}We denote by $\mathcal{S}_h^\prime(\mathbb{R}^d)$ the set of tempered distribution $u$ such that\begin{equation*}
		\lim_{j\rightarrow-\infty}\|\chi(2^{-j}D)u\|_{L^\infty}=0,
	\end{equation*}then the following decomposition holds for all $u\in\mathcal{S}_h^\prime(\mathbb{R}^d)$:\begin{equation*}\begin{split}
			u=\sum_{j\in\mathbb{Z}}\dot{\Delta}_j u.
		\end{split}
	\end{equation*}\par
	Based on the above decomposition, we give the definition of homogeneous Besov spaces as follows.\begin{definition}\label{besovdef}
		Let $s\in\mathbb{R}$ and $1\leq p,r\leq\infty$, the homogeneous Besov space $\dot{B}^{s}_{p,r}(\mathbb{R}^d)$ is defined by\begin{equation*}
			\dot{B}^{s}_{p,r}(\mathbb{R}^d):=\{u\in\mathcal{S}^\prime_h(\mathbb{R}^d):\|u\|_{\dot{B}^s_{p,r}}<\infty\}
		\end{equation*}with\begin{equation*}
			\|u\|_{\dot{B}^s_{p,r}}:=\|(2^{js}\|\dot{\Delta}_j u\|_{L^p})_{j\in\mathbb{Z} }\|_{\ell^r}.
		\end{equation*}
	\end{definition}
	The following well-known Berstein's inequalities will be frequently used in this sequel.\begin{lemma}\label{bernstein}
		Let $0<r<R$. There exists a constant $C$ such that for any nonnegative integer $k$, any couple $(p,q)\in[1,\infty]^2$ with $q\geq p\geq1$, and any function $u$ in $L^p(\mathbb{R}^d)$, there holds\begin{equation*}
			\begin{split}
				&\operatorname{supp}\mathcal{F}u\subset\{\xi\in\mathbb{R}^d:|\xi|\leq\lambda R \}\Longrightarrow\sup_{|\alpha|=k}\|\partial^\alpha u\|_{L^q}\leq C^{k+1}\lambda^{k+d(\frac{1}{p}-\frac{1}{q})}\|u\|_{L^p},\\
				&\operatorname{supp}\mathcal{F}u\subset\{\xi\in\mathbb{R}^d:\lambda r\leq|\xi|\leq\lambda R \}\Longrightarrow C^{-k-1}\lambda^{k}\|u\|_{L^p}\leq\sup_{|\alpha|=k}\|\partial^\alpha u\|_{L^p}\leq C^{k+1}\lambda^{k}\|u\|_{L^p}.
			\end{split}
		\end{equation*}
	\end{lemma}
	As an immediate consequence of Definition \ref{besovdef} and
	Lemma \ref{bernstein}, we have the following statement.
	\begin{lemma}\label{prop} Let $s\in\mathbb{R}$ and $1\leq p,r\leq \infty$, the following properties hold true:
		\begin{enumerate}
			\item[(i)] Derivatives: for any $k$-th order derivative, it holds that\begin{equation*}
				\sup_{|\alpha|=k}\|\partial^\alpha u\|_{\dot{{B}}^{s}_{p,r}}\approx \|u\|_{\dot{{B}}^{s+k}_{p,r}}.
			\end{equation*}
			\item[(ii)] Embedding: for $1\leq p\leq \tilde{p}\leq \infty$ and $1\leq r\leq \tilde{r}\leq \infty$, one has\begin{equation*}\begin{split}
					&\dot{{B}}^{s}_{p,r}(\mathbb{R}^d)\hookrightarrow \dot{{B}}^{s-d(\frac{1}{p}-\frac{1}{\tilde{p}})}_{\tilde{p},\tilde{r}}(\mathbb{R}^d),\quad \dot{{B}}^{\frac{d}{p}}_{p,1}(\mathbb{R}^d)\hookrightarrow L^\infty(\mathbb{R}^d),\\
					&\quad\text{and}\quad \dot{{B}}^{0}_{p,1}(\mathbb{R}^d)\hookrightarrow L^p(\mathbb{R}^d).
				\end{split}
			\end{equation*}
			\item[(iii)] Interpolation: for any  $s_1<s_2$ and $0<\theta<1$, there holds\begin{align}
				&\|u\|_{\dot{B}^{\theta s_1+(1-\theta)s_2}_{p,r}}\leq \|u\|^\theta_{\dot{B}^{s_1}_{p,r}}\|u\|^{1-\theta}_{\dot{B}^{s_2}_{p,r}},\label{inter1}\\
				&\|u\|_{\dot{B}^{\theta s_1+(1-\theta)s_2}_{p,1}}\leq\frac{C}{s_2-s_1}\Big(\frac{1}{\theta}+\frac{1}{1-\theta}\Big) \|u\|^\theta_{\dot{B}^{s_1}_{p,\infty}}\|u\|^{1-\theta}_{\dot{B}^{s_2}_{p,\infty}}.\label{inter2}
			\end{align}
			\item [(iv)] Action of multiplier: Let $f$ be a smooth function on $\mathbb{R}^d\backslash\{0\}$ which is homogeneous of degree $m$, then we have\begin{equation*}
				\|f(D)u\|_{\dot{B}^{s-m}_{p,r}}\lesssim\|u\|_{\dot{B}^{s}_{p,r}}.
			\end{equation*}
		\end{enumerate}		
	\end{lemma}
	To estimate the time-dependent function valued in Besov space, we need the following mixed time- space introduced by Chemin and Lerner \cite{MR1354312}.
	\begin{definition}
		Let $T>0$, $s\in\mathbb{R}$ and $1\leq p,r,q\leq\infty$, the space $\widetilde{L}^q(0,T;\dot{B}^s_{p,r}(\mathbb{R}^d))$ is defined by\begin{equation*}
			\begin{split}
				\widetilde{L}^q(0,T;\dot{B}^s_{p,r}(\mathbb{R}^d)):=\{u\in L^q(0,T;\mathcal{S}^\prime_h(\mathbb{R}^d)):\|u\|_{\widetilde{L}^q_T(\dot{B}^s_{p,r})}<\infty\}
			\end{split}
		\end{equation*}with\begin{equation*}
			\|u\|_{\widetilde{L}^q_T(\dot{B}^s_{p,r})}:=\|(2^{js}\|\dot{\Delta}_j u\|_{L_T^q(L^p)})_{j\in\mathbb{Z}}\|_{\ell^r}.
		\end{equation*}
	\end{definition}
	\begin{remark}\upshape\label{minkow}
		By the Minkowski's inequality, one easily observe that\begin{equation}
			\begin{cases}
				\|u\|_{\widetilde{L}^q_T(\dot{B}^s_{p,r})}\leq\|u\|_{{L}^q_T(\dot{B}^s_{p,r})}\quad \text{if}\quad q\leq r,  \\
				\|u\|_{{L}^q_T(\dot{B}^s_{p,r})}\leq\|u\|_{\widetilde{L}^q_T(\dot{B}^s_{p,r})}\quad \text{if}\quad q\geq r.
			\end{cases}
		\end{equation}
	\end{remark}
	For any $u\in\mathcal{S}_h^\prime(\mathbb{R}^d)$, we designate $u^\ell$ and $u^h$ as the low-frequency part and high-frequency part of $u$ respectively, that is\begin{equation*}
		\begin{split}
			u^\ell:=\sum_{j\leq -1}\dot{\Delta}_j u,\quad u^h:=\sum_{j\geq 0}\dot{\Delta}_j u.
		\end{split}
	\end{equation*}Besides, in order to restrict the Besov norms to the low and high frequencies, we adopt the following notations:\begin{equation*}
		\begin{split}
			\|u\|_{\dot{B}^s_{p,r}}^\ell&:=\|(2^{js}\|\dot{\Delta}_j u\|_{L^p})_{j\leq 0}\|_{\ell^r},\quad \|u\|_{\dot{B}^s_{p,r}}^h:=\|(2^{js}\|\dot{\Delta}_j u\|_{L^p})_{j\geq -1}\|_{\ell^r},\\
			\|u\|_{\widetilde{L}^q_T(\dot{B}^s_{p,r})}^\ell&:=\|(2^{js}\|\dot{\Delta}_j u\|_{L_T^q(L^p)})_{j\leq 0}\|_{\ell^r},\quad \|u\|_{\widetilde{L}^q_T(\dot{B}^s_{p,r})}^h:=\|(2^{js}\|\dot{\Delta}_j u\|_{L_T^q(L^p)})_{j\geq -1}\|_{\ell^r}.
		\end{split}
	\end{equation*}
	It is easy to check that for any $s^\prime>0$,\begin{equation}\label{lhembedding}
		\begin{cases}
			\|u^\ell\|_{\dot{B}^s_{p,r}}\lesssim	\|u\|_{\dot{B}^s_{p,r}}^\ell\lesssim \|u\|_{\dot{B}^{s-s^\prime}_{p,r}}^\ell,\\
			\|u^h\|_{\dot{B}^s_{p,r}}\lesssim	\|u\|_{\dot{B}^s_{p,r}}^h\lesssim \|u\|_{\dot{B}^{s+s^\prime}_{p,r}}^h.
		\end{cases}
	\end{equation}
	\par  Next, we gather some product and commutator estimates, which play a fundamental role in the analysis of the nonlinear terms.
	\begin{lemma}
		The following product estimates hold true:\begin{itemize}
			\item Let $s>0$ and $1\leq p,r\leq \infty$, it holds that\begin{equation}\label{law}
				\begin{split}
					\norm{fg}_{\dot{B}_{p,r}^s}\lesssim \norm{f}_{L^\infty}\norm{g}_{\dot{B}_{p,r}^s}+\norm{g}_{L^\infty}\norm{f}_{\dot{B}_{p,r}^s}.
				\end{split}
			\end{equation}
			\item Let $2\leq p\leq \infty$, $s_1\leq\frac{d}{p}$, $s_2\leq\frac{d}{p}$ and $s_1+s_2>0$, it holds that\begin{equation}\label{daishu}
				\begin{split}
					\|fg\|_{\dot{B}_{p,1}^{s_1+s_2-\frac{d}{p}}}\lesssim \|f\|_{\dot{B}_{p,1}^{s_1}}\|g\|_{\dot{B}_{p,1}^{s_2}}.
				\end{split}
			\end{equation}
			\item Let $2\leq p\leq \infty$, $s_1\leq\frac{d}{p}$, $s_2<\frac{d}{p}$ and $s_1+s_2\geq0$, it holds that\begin{equation}\label{guangyidaishu}
				\begin{split}
					\|fg\|_{\dot{B}_{p,\infty}^{s_1+s_2-\frac{d}{p}}}\lesssim \|f\|_{\dot{B}_{p,1}^{s_1}}\|g\|_{\dot{B}_{p,\infty}^{s_2}}.
				\end{split}
			\end{equation}
		\end{itemize}
	\end{lemma}
	
	\begin{lemma}
		Let $1\leq p\leq \infty$ and $\w$ be an vector field over $\mathbb{R}^d$, then the following commutator estimates hold true:\begin{itemize}
			\item If $-\frac{d}{p}<s\leq\frac{d}{p}+1$, one has\begin{equation}\label{jiaohuanzi1}
				\begin{split}
					\sum_{j\in\mathbb{Z}}2^{js}\|[\dot{\Delta}_j,\w\cdot\nabla]f\|_{L^p}\lesssim \|\w\|_{\dot{B}^{\frac{d}{p}+1}_{p,1}}\|f\|_{\dot{B}^{s}_{p,1}}.
				\end{split}
			\end{equation}
			\item If $s>0$, one has\begin{equation}\label{jiaohuanzi2}
				\begin{split}
					\sum_{j\in\mathbb{Z}}2^{js}\|[\dot{\Delta}_j,\w\cdot\nabla]f\|_{L^p}\lesssim \|\nabla \w\|_{L^\infty}\|f\|_{\dot{B}^{s}_{p,1}}+\|\nabla f\|_{L^\infty}\|\w\|_{\dot{B}^{s}_{p,1}}.
				\end{split}
			\end{equation}
			\item If $-\frac{d}{p}\leq s<\frac{d}{p}+1$, one has\begin{equation}\label{jiaohuanzi3}
				\begin{split}
					\sup_{j\in\mathbb{Z}}2^{js}\|[\dot{\Delta}_j,\w\cdot\nabla]f\|_{L^p}\lesssim \|\w\|_{\dot{B}^{\frac{d}{p}+1}_{p,1}}\|f\|_{\dot{B}^{s}_{p,\infty}}.
				\end{split}
			\end{equation}
		\end{itemize}
	\end{lemma}
	Finally, we need the following lemma to bound the composition functions.
	\begin{lemma}\label{fuhe}
		Let $G$ be a smooth function on $\mathbb{R}$ with $G(0)=0$. Let $s>0,$ $1\leq p,r\leq \infty$. Then, for any function
		$f\in \dot{B}^{s}_{p,r}(\mathbb{R}^d)\cap L^\infty(\mathbb{R}^d)$, we have  $G(f)\in \dot{B}^{s}_{p,r}(\mathbb{R}^d)\cap L^\infty(\mathbb{R}^d)$ and\begin{equation}
			\begin{split}
				\|G(f)\|_{\dot{B}^{s}_{p,r}}\leq C\|f\|_{\dot{B}^{s}_{p,r}},
			\end{split}
		\end{equation}where the constant $C>0$ depends only on $\|f\|_{L^\infty}$, $G^\prime$, $s$ and $d$.
	\end{lemma}
	
	\section*{Acknowledgments}
	Zhai was partially supported by the Guangdong Provincial Natural Science
	Foundation under grant 2024A1515030115 and 2022A1515011977.
	\bigskip
	
	\noindent{\bf Conflict of interest}
	On behalf of all authors, the corresponding author states that there is no conflict of interest.
	\vskip 0.4 true cm
	\noindent{\bf Data availability statement}
	Data sharing not applicable to this article as no datasets were generated 
	current study.

	\bibliographystyle{abbrv}
	
\end{document}